\documentclass[11pt, reqno]{amsart}

\usepackage{amssymb,latexsym,amsmath,amsfonts}
\usepackage{mathrsfs}
\usepackage{graphicx}
\usepackage{upgreek}
\usepackage[usenames]{color}
\usepackage{tikz-cd}
\usepackage{hyperref}
\usepackage{comment}
\usepackage{enumitem}

\numberwithin{equation}{section}

\allowdisplaybreaks

\theoremstyle{definition}
\newtheorem{definition}{Definition}[section]

\theoremstyle{remark}
\newtheorem{remark}[definition]{Remark}

 \theoremstyle{plain}
\newtheorem{theorem}[definition]{Theorem}
\newtheorem{result}[definition]{Result}
\newtheorem{lemma}[definition]{Lemma}
\newtheorem{proposition}[definition]{Proposition}

\newtheorem{corollary}[definition]{Corollary}

\newcommand{\lam}{\lambda}
\newcommand{\eps}{\epsilon}
\newcommand{\zt}{\zeta}

\newcommand{\OM}{\Omega}
\newcommand{\D}{\mathbb{D}}

\newcommand{\hol}{\mathcal{O}}

\newcommand{\Sn}{S_n(\Omega)}
\newcommand\sym[1]{\Sigma^{#1}(\Omega)}
\newcommand\intgR[2]{[{#1}\,.\,.\,{#2}]}
\newcommand{\lrarw}{\longrightarrow}
\newcommand{\ch}{{\bf c}}
\newcommand\bv[1]{{#1}^\bullet}

\newcommand{\nat}{\mathbb{N}}
\newcommand{\I}{\mathbb{I}}

\newcommand{\bcdot}{\boldsymbol{\cdot}}

\newcommand\intf[4]{\genfrac{#1}{#2}{0.5pt}{0}{#3}{#4}}

\newcommand{\mats}{\mathcal{S}}

\newcommand{\C}{\mathbb{C}} 
\newcommand\mfd[1]{\mathscr{#1}}
\makeatletter
\newcommand*{\rom}[1]{\expandafter\@slowromancap\romannumeral #1@}
\makeatother

\begin{document}

\title[Spectral Cartan Theorem]{On a spectral version of Cartan's theorem}

\author{Sayani Bera, Vikramjeet Singh Chandel and Mayuresh Londhe}
\address{Indian Association for the Cultivation of Science, Kolkata 700032, India}
\email{sayanibera2016@gmail.com}
\address{Harish-Chandra Research Institute, Prayagraj (Allahabad) 211019, India}
\email{abelvikram@gmail.com, vikramjeetchandel@hri.res.in}
\address{Department of Mathematics, Indian Institute of Science, Bangalore 560012, India}
\email{mayureshl@iisc.ac.in}



\begin{abstract}
For a domain $\Omega$ in the complex plane,
we consider the domain $S_n(\Omega)$ consisting of those $n\times n$ complex matrices whose spectrum is 
contained in $\Omega$. Given a holomorphic self-map $\Psi$ of $S_n(\Omega)$ such that
$\Psi(A)=A$ and the derivative of $\Psi$ at $A$ is identity for some $A\in S_n(\Omega)$, we investigate
when the map $\Psi$ would be spectrum-preserving. We prove that
if the matrix $A$ is either diagonalizable or non-derogatory then for {\em most} domains $\Omega$,
$\Psi$ is spectrum-preserving on $S_n(\Omega)$.
Further, when $A$ is arbitrary,
we prove that $\Psi$ is spectrum-preserving on a certain analytic subset of $S_n(\Omega)$.

\end{abstract}
\maketitle

\section{Introduction and statement of main results}\label{S:intro}

A well-known result of Cartan about holomorphic self-maps, also known as Cartan's uniqueness theorem,
says:
every holomorphic self-map of a bounded domain (in the complex Euclidean space) that has a fixed point so that the
derivative of the holomorphic map at the fixed point is identity has to be the identity map on the given bounded domain.
The above result was generalized to taut complex manifolds by Wu \cite{Wu:Normfamily67}
and shortly later to the case of Kobayashi hyperbolic complex manifolds by Kobayashi \cite{Kobayashi:hm&hm70}.
The purpose of this article is to explore holomorphic
self-maps of certain matricial domains\,---\,that are not Kobayashi hyperbolic (thus not taut)\,---\,in the spirit of Cartan's Theorem.
We begin with introducing these domains.
\smallskip

Given $n\in\nat$, $n\geq 2$, we denote by $M_n(\C)$ the set of all $n\times n$ complex matrices. 
For a matrix $W\in M_n(\C)$, the 
{\em spectrum} of $W$ is the set of eigenvalues of $W$ and is denoted by $\sigma(W)$.
Let $\OM$ be a domain in the complex plane $\C$, we consider the set 
\[
 \Sn\,:=\,\{W\in M_n(\C)\,:\,\sigma(W)\subset\OM\}.
\]
Note that $\Sn$ is an open and connected subset of $M_n(\C)\equiv\C^{n^2}$. 
In the case when $\OM=\D$, where $\D$ is the open unit disc in $\C$ 
centered at the origin, the domain $S_n(\D)$ is called the
{\em spectral unit ball}. 
In \cite{Ransford-White: holselfspectunitball91}, Ransford--White initiated function-theoretic study of the spectral unit ball. 
Since then the spectral unit ball
has been studied intensively in the literature; see, for instance, \cite{AndristKutz:fdpautspecball},
\cite{Bharali:obsintpspecball07}, \cite{CostaraRansford:locirredspect07},
\cite{NikolovPflugThomas:SpecNevancaraFejer11}, \cite{wz:phmsub08} and the refrences therein.
We now enlist an important observation about the domains $\Sn$.
\begin{lemma}\label{L:tautkobanot}
For any domain $\OM\subset\C$ and $n\geq 2$, the domain $\Sn$ is not Kobayashi hyperbolic.
\end{lemma}
In fact, for every $\Sn$, $n\geq 2$, and for any $W\in\Sn$ there exists a {\bf non-constant} holomorphic map $f_W:\C \lrarw \Sn$
such that $\sigma(f(\bcdot))=\sigma(W)$ on $\C$.
We postpone the proof of Lemma~\ref{L:tautkobanot} to Section~\ref{S:Basicprelim}, where we also recall several
relevant definitions and results.
Thus given a holomorphic self-map
$\Psi$ of $\Sn$, $n\geq 2$, such that $\Psi(A)=A$ and
the derivative of $\Psi$ at $A$ is identity, i.e., $\Psi'(A)=\I$, the aforementioned 
results of Kobayashi and Wu
cannot be directly applied to conclude that $\Psi$ is the identity map on $\Sn$. Indeed, there exists a holomorphic 
self-map $\Psi$ of $S_2(\D)$ such that $\Psi(0)=0$ and $\Psi'(0)=\I$ that is not even injective
(see \cite[Section~0]{Ransford-White: holselfspectunitball91}).
\smallskip

To study the holomorphic self-maps of $\Sn$, we employ its relation with the $n$-th symmetrized product of $\OM$
which, in general, have many nice properties. One such important property
is that\,---\,while none of the domains $\Sn$ are Kobayashi hyperbolic\,---\,the $n$-th symmetrized product of $\OM$
is Kobayashi hyperbolic for most domains $\OM\subset\C$. To define this latter object,
we consider the {\em symmetrization map} $\pi_n:\C^n\lrarw\C^n$ defined by
$\pi_n(z):=\big(\pi_{n,\,1}(z),\dots,\pi_{n,\,j}(z),\dots,\pi_{n,\,n}(z)\big)$,
where $\pi_{n,\,j}(z)$ is the $j$-th elementary symmetric polynomial in
variables $z_1,\dots,z_n$ for $z:=(z_1,\dots,z_n)$.
In other words, we have
\[
\prod_{j=1}^n(t-z_j)=t^n+\sum_{j=1}^n (-1)^j\pi_{n,\,j}(z_1,\dots,z_n)\,t^{n-j}, \ \ \ t\in\C.
\]
The {\em $n$-th symmetrized product of $\Omega$}, denoted by $\Sigma^{n}(\Omega)$, is defined by
$\Sigma^{n}(\Omega):=\pi_n(\Omega^n)$. Since the symmetrization map $\pi_n:\C^n\lrarw\C^n$ is a proper holomorphic map,  
it follows that $\Sigma^{n}(\Omega)$ is a domain in $\C^n$.
\smallskip

The aforementioned relation between $\Sn$ and $\Sigma^{n}(\Omega)$
is via the map $\ch:M_n(\C)\longrightarrow\C^n$ defined by $\ch(W):=\big(c_1(W),\dots,c_n(W)\big)$ where
the polynomial 
\[
 t^n+\sum_{k=1}^n (-1)^{k}\,c_k(W)t^{n-k}
 \]
is the characteristic polynomial of $W$. 
We shall denote the restriction of the map $\ch$ to any open subset of $M_n(\C)$ by $\ch$ itself.
Observe for each $k$, $1\leq k\leq n$, since $c_k(W)$ is the sum of all principal minors of order $k$ of the matrix $W$, 
$\ch$ is a holomorphic map on $M_n(\C)$.
Further,
if $\{\lam_1,\dots,\lam_n\}$ is a list of eigenvalues of $W$,
repeated with algebraic multiplicity, then $c_k(W)=\pi_{n,\,k}(\lam_1,\dots,\lam_n)$ for each $k$, where $c_k(W)$ 
is the $k$-th coordinate
of $\ch(W)$. It now follows that
\[
\ch(\Sn)=\Sigma^n(\Omega) \ \ \text{and} \ \ \Sn=\ch^{-1}(\Sigma^n(\Omega))
\] 
for any domain $\OM \subset \C$.
It turns out that the domain $\Sigma^n(\Omega)$ is Kobayashi hyperbolic (and also Kobayashi complete)
if and only if the cardinality 
of $\C\setminus \Omega$ is at least $2n$ (see Result~\ref{Res:kobhypsymprod}).
In Section~\ref{S:prelim_results}, using the Kobayashi hyperbolicity of $\Sigma^n(\OM)$, we prove that 
every holomorphic self-map $\Psi$ of $\Sn$ induces a unique holomorphic self-map $G_{\Psi}$ of $\Sigma^n(\Omega)$ such that 
$\ch\circ\Psi=G_{\Psi}\circ\ch$, i.e., the following diagram commutes:
\vspace{-0.3cm}
\begin{figure}[h!]
\[
\begin{tikzcd}
 \Sn \arrow[rr, "\Psi"] 
  \arrow[d, "\ch"]
  & 
  & \Sn \arrow[d, "\ch"] \\
  \Sigma^n(\Omega) \arrow[rr, "G_{\Psi}"]
  &
  & \Sigma^n(\Omega)\\
\end{tikzcd}
\] 
\vspace{-1.1cm}
\caption{}
\label{F:comd1}
\end{figure}

By studying the map $G_{\Psi}$, we are led to the first main result of this article.  

\begin{theorem}\label{T:mainT1}
Given $n\in\nat$, $n\geq 2$, and a domain $\Omega \subset \C$ satisfying $\#(\C\setminus \Omega) \geq 2n$. 
Let $\Psi$ be a holomorphic self-map of $\Sn$ such that $\Psi(A)=A$ and $\Psi'(A)=\I$ for some $A \in \Sn$. 
Assume that the matrix $A$ is either
a diagonalizable matrix or a non-derogatory matrix. Then
\[
 \ch(\Psi(W))=\ch(W) \ \ \text{for every} \ \ W\in\Sn.
 \]
Consequently, $\sigma\big(\Psi(W)\big)=\sigma\big(W\big)$ and the 
algebraic multiplicity of each eigenvalue is preserved for every $W\in\Sn$,
i.e., $\Psi$ is spectrum-preserving on $\Sn$. 
\end{theorem}

\begin{remark}\label{Rem:dense}
Recall that a non-derogatory matrix is a matrix for which the characteristic polynomial and the minimal polynomial are 
same, see \cite[p.195]{hornJohn:matanal85} for other equivalent definitions. 
Observe that the set of diagonalizable matrices is dense in $\Sn$, and the set of non-derogatory matrices is open 
and dense in $\Sn$ for any domain $\OM$ in $\C$. Furthermore, given any $A\in S_2(\OM)$, it is either
a diagonalizable matrix or a non-derogatory matrix. Hence, when $n=2$, the condition on the matrix $A$ in Theorem~\ref{T:mainT1}
is superfluous.
\end{remark}

As an application of Theorem~\ref{T:mainT1}, we prove a result that
gives stronger conclusion in a neighbourhood of the matrix $A$ than that of Theorem~\ref{T:mainT1}.

\begin{corollary}\label{cor:locconj}
Let $n,\,\OM,\,\Psi$ and $A$ be as in Theorem~\ref{T:mainT1}. Then there exists
a neighbourhood $\mathcal{N}$ of
$A$ such that $\Psi(W)$ is conjugate to $W$ for any $W\in\mathcal{N}$.
\end{corollary}

\begin{remark}\label{Rm:commentransfordwhite}
Note that without any condition on $\Omega$ in Theorem~\ref{T:mainT1}, 
$\Psi$ need not be spectrum-preserving. For example: let $\Omega=\C\setminus \{0\}$
and consider $\Psi(W):= \exp(W- \I)$. Notice that
$\Psi$ satisfies $\Psi(\I)=\I$ and $\Psi'(\I)=\I$
but $\Psi$ is not spectrum-preserving. When $\OM=\D$ and $A=0\in S_n(\D)$, $n\geq 2$, 
the above theorem was proved by Ransford--White \cite[Theorem~3]{Ransford-White: holselfspectunitball91}.
Since the automorphism group of $S_n(\D)$ is far from being transitive (see
\cite[Theorem~4]{Ransford-White: holselfspectunitball91}), one cannot use the result due to Ransford--White to 
deduce the conclusion of the above theorem
for $S_n(\D)$ and for an arbitrary $A \neq 0$.%
\end{remark}

We must mention that our proof of Theorem~\ref{T:mainT1} is not a routine extension of the argument
given by Ransford--White in the case mentioned above. 
We provide a very short sketch of the proof of Theorem~\ref{T:mainT1} to point out some features of it
that are novel.
Let $\Psi$ be as in Theorem~\ref{T:mainT1} and $G_{\Psi}$ be the self-map of $\Sigma^n(\Omega)$ associated with
$\Psi$ (see Figure~\ref{F:comd1}).
Observe that if we show that
$G_{\Psi}$ is the identity map on $\Sigma^n(\Omega)$ then Theorem~\ref{T:mainT1} follows.
To study the set of fixed points of $G_{\Psi}$,
we introduce the technique of local decomposition of the map $\ch$,
which is particularly useful when the matrix $A$ is non-zero.
\begin{itemize}[leftmargin=14pt]
\item Let $\lam_i$ be an eigenvalue of $A$ with algebraic multiplicity 
$n_i$, $i\in\{1,2, \dots, m\}$. Then the map $\ch$
decomposes {\bf locally} as $\ch\,=\,\tau\circ\theta$, where $\theta$ is a map into the cartesian product of 
$\sym{n_i}$, i.e., $\prod_{i=1}^m\sym{n_i}$ and $\tau$ is the canonical map from $\prod_{i=1}^m\sym{n_i}$ 
onto $\sym{n}$. 
Furthermore, we observe that $\tau$ is a local biholomorphism. This allows us to define 
the holomorphic map $F_{\Psi}$\,---\,in a neighbourhood of $\theta(A)$\,---\,which is locally a 
biholomorphic conjugate of $G_{\Psi}$ via the map $\tau$ (see Figure~\ref{F:comd2}).
Note this step is independent of the choice of the matrix $A$. 
\smallskip

\item When $A$ is diagonalizable, using a result on
the perturbation of eigenvalues of a normal matrix by Sun \cite{Sun:Varspectnormal96}, 
we prove that the trace of $F'_{\Psi}$ at $\theta(A)$ is $n$, where 
$F_{\Psi}$ is as above. This implies that the trace of $G'_{\Psi}$ at $\ch(A)$ is $n$.
On the other hand, when $A$ is non-derogatory, we 
explicitly construct a right inverse of the map $\ch$ passing through the point $A$ which, in particular, shows that $G'_{\Psi}(\ch(A))=\I$.
\smallskip

\item Under the cardinality condition  $\#(\C\setminus \Omega) \geq 2n$, the domain $\Sigma^n(\OM)$ is Kobayashi complete. 
We appeal to results from the
iteration theory of holomorphic self-maps on taut complex manifolds (see Section~\ref{S:Basicprelim})\,---\,together with the information about 
$G'_{\Psi}(\ch(A))$ above\,---\,to establish that $G_{\Psi}$ is the identity map on $\Sigma^n(\Omega)$.
\end{itemize}

\begin{remark}
Sun's result\,---\,alluded to as above\,---\,gives a bound on the eigenvalues for the perturbation of a normal matrix
(see Result~\ref{Res:sun} for the statement). In general, similar bounds on eigenvalues for perturbation of an arbitrary matrix
have robust error; see, for instance, \cite{Song:varspecarbmat02}, \cite{Bhatia:Perboundmateigen07}.
Further, given a matrix $B$ if there exists a {\bf local} right inverse of the map $\ch$ passing through the point $B$, then $B$ has to be a non-derogatory 
matrix.  Therefore, the techniques used to prove Theorem~\ref{T:mainT1} do not extend.
\end{remark}

We now turn to the case when the matrix $A$ in Theorem~\ref{T:mainT1} is not necessarily diagonalizable or non-derogatory. 
Notice, if we take derivatives on both sides of $\ch\circ\Psi\,=\,G_{\Psi}\circ\ch$ at $A$ then it follows that 
the range space of the derivative of $\ch$ at $A$ lies in the eigenspace of 
the derivative of $G_{\Psi}$ at $\ch(A)$ corresponding to the eigenvalue $1$. 
This eigenspace plays an important role with regards to the spectrum-preserving property of $\Psi$. 
Now, the rank of $\ch$ at $A$ gives a lower bound on the dimesion of the eigenspace of 
$G'_{\Psi}(\ch(A))$ corresponding to the eigenvalue $1$. In this direction, we have the following proposition,
which is interesting in its own right:

\begin{proposition}\label{P:rankofcharmap}
Let $A\in M_n(\C)$ be given. Then the rank of the derivative of $\ch$ at $A$ is equal to
the degree of the minimal polynomial of $A$.
\end{proposition}

Our proof of Proposition~\ref{P:rankofcharmap} crucially uses the local decomposition of $\ch$ as described before. 
By a result of Vigu\'{e} \cite{Vigue:FPlhm} (see Section~\ref{S:Basicprelim}) about the fixed point set of holomorphic self-maps,
the eigenspace of $G'_{\Psi}(\ch(A))$ corresponding to the eigenvalue $1$ 
determines the fixed-point set of $G_{\Psi}$.
Using Proposition~\ref{P:rankofcharmap}\,---\,in a way that is described in the last paragraph\,---\,we get a lower 
bound on the dimesion of the fixed-point set of $G_{\Psi}$ which leads to the second main result of this article.

\begin{theorem}\label{T:fixedpoint}
Given $n\in\nat$, $n\geq 2$, and a domain $\Omega \subset \C$ satisfying $  \#(\C\setminus \Omega) \geq 2n$. 
Let $\Psi$ be a holomorphic self-map of $\Sn$ such that $\Psi(A)=A$ and $\Psi'(A)=\I$. Then 
there is a closed complex submanifold $\mathscr{S}$ of $\Sigma^n(\Omega)$
containing $\ch(A)$ of complex dimension greater than or equal to the degree of the minimal polynomial
of $A$ such that for every $W\in{\ch}^{-1}(\mathscr{S})$ we have $\ch(\Psi(W))=\ch(W)$.
\end{theorem}

Since for a non-derogatory matrix the degree of the minimal polynomial is maximal, Theorem~\ref{T:fixedpoint} 
gives an alternate proof of Theorem~\ref{T:mainT1} when $A$ is non-derogatory.
We prove Theorem~\ref{T:mainT1} and Theorem~\ref{T:fixedpoint} in Section~\ref{S:diagcase} and 
Section~\ref{S:nonderog} respectively while the proof of Proposition~\ref{P:rankofcharmap} is given in 
Section~\ref{S:rankcomp}.
\medskip

\noindent{{\bf{Concluding remarks:}}} For a domain $\OM$ with $\#(\C\setminus\OM)\geq 2n$, the map
$\Psi$ as in Theorem~\ref{T:mainT1} is spectrum-preserving when the matrix $A$ belongs to a large subset of 
$\Sn$ (see Remark~\ref{Rem:dense}). Thus it seems that the same conclusion should hold for any choice of $A$, but current 
tools and results are not enough to conclude this. For example,
when $n=3$, there is one particular choice of the matrix $A$ for which we are not able to say
whether $\Psi$ is spectrum-preserving (see Section~\ref{S:nonderog}). 
It would be interesting to find a counterexample in this case.

\section{Kobayashi hyperbolicity and iteration theory on taut complex manifolds}\label{S:Basicprelim}

In this section, we recall notions of Kobayashi hyperbolicity, Kobayashi completeness and tautness for a given complex manifold. 
As hinted in Section~\ref{S:intro}, we shall need 
results from the iteration theory of holomorphic self-maps on taut complex manifolds in our proofs, so we state those results too in this section.
Before we begin, a piece of notation\,---\,given complex manifolds $\mfd{X}$ and $\mfd{Y}$, we shall denote by 
$\hol(\mfd{X},\,\mfd{Y})$ the set of all holomorphic maps from $\mfd{X}$ into $\mfd{Y}$.
\smallskip

Let $\mfd{X}$ be a complex manifold and let ${\sf h}$ denote the hyperbolic distance induced by
the Poincar\'{e} metric on the unit disc $\D$. 
The {\em Kobayashi pseudo-distance} $K_{\mfd{X}}:\mfd{X}\times \mfd{X}\lrarw [0,\,\infty)$ is defined by:
given two points $p,q\in\mfd{X}$,
\[
K_{\mfd{X}}(p, q):=\inf\Big\{\sum_{i=1}^k {\sf h}(\zt_{i-1},\zt_i)\,:\,(\phi_1,\dots,\phi_k;\zt_0,\dots,\zt_k)\in
\mathcal{A}(p, q)\Big\},
\]
where $\mathcal{A}(p, q)$ is the set of all analytic chains in $\mfd{X}$ joining $p$ to $q$. Here,
$(\phi_1,\dots,\phi_k;\zt_0,\dots,\zt_k)$ is an {\em analytic chain} in $\mfd{X}$ joining $p$ to $q$ if 
$\phi_i\in\hol(\D,\,\mfd{X})$ for each $i$ such that
\[
p=\phi_1(\zt_0), \ \ \phi_k(\zt_k)=q \ \ \text{and} \ \ \phi_i(\zt_i)=\phi_{i+1}(\zt_i)
\]
for $i=1,\dots,k-1$.
\smallskip

It is not difficult to check that $K_{\mfd{X}}$ is a pseudo-distance. Using the Schwarz lemma on the unit disc $\D$,
we see that $K_{\D}\equiv{\sf h}$. An important property of the Kobayashi pseudo-distance is its contractivity under holomorphic maps, i.e.,
if $F: \mfd{X}\lrarw \mfd{Y}$ is a holomorphic map then
$K_{\mfd{Y}}\big(F(p), F(q)\big)\leq K_{\mfd{X}}(p, q)$
for all $p,q\in \mfd{X}$.
A complex manifold $\mfd{X}$ is called {\em Kobayashi hyperbolic} if the pseudo-distance $K_{\mfd{X}}$ is 
a distance, i.e., $K_{\mfd{X}}(p, q)=0$ if and only if $p=q$. Furthermore, $\mfd{X}$ is called {\em Kobayashi complete} if it is Kobayashi 
hyperbolic and the metric space $(\mfd{X},\,K_\mfd{X})$ is complete. It is a fact that
every bounded domain in $\C^d$ is Kobayashi hyperbolic.
On the other hand, it is easy to check that $K_{\C^d}\equiv 0$ for all $d\geq 1$.
We refer the interested reader to \cite{Kobayashi:Hypcomplex98} (also see \cite[Chapter~3]{JP13:jarnicki2013invariant}) for a comprehensive 
account on Kobayashi pseudo-distance.
Now, we recall the following generalization of 
Liouville's theorem.
\begin{result}\label{Res:LT}
Let $\mfd{X}$ be a Kobayashi hyperbolic complex manifold and let $F:\C^d\lrarw\mfd{X}$ be a holomorphic map. Then $F$ is a constant function.
\end{result}

We are now in a befitting position to present
\begin{proof}[The proof of Lemma~\ref{L:tautkobanot}]
{\bf Fix} $\Omega \subset \C$ and $n \geq 2$. 
Now consider a point $W\in\Sn$.
Let $D$ be a diagonal matrix such that $\ch(D)=\ch(W)$. We know that there exists $C\in M_n(\C)$ and a
strictly upper triangular matrix $U$ such that
\[
W=\exp(-C)\,(D+U)\,\exp(C).
\]
Now consider the map $f:\C\lrarw M_n(\C)$ defined by 
\[
f(\zt):=\exp(-C\,\zt)\,(D+\zt\,U)\,\exp(C\,\zt) \ \ \ \forall\zt\in
\C.
\]
Note that $\ch(f(\zt))=\ch(D+\zt\,U)=\ch(D)$, hence $f(\C)\subset\Sn$. Since $f$ is a non-constant holomorphic map into $\Sn$,
by Result~\ref{Res:LT}, $\Sn$ cannot be Kobayashi hyperbolic.
\end{proof}

Recall, a complex manifold $\mfd{X}$ is called {\em taut} if every 
sequence in $\hol(\D,\,\mfd{X})$ either has a convergent subsequence or a compactly divergent subsequence. 
It is a fact that every Kobayashi complete complex manifold is taut and every taut complex manifold is Kobayashi hyperbolic;
the converse of both these facts do not hold.
We now state the relevant results from the iteration theory of holomorphic self-maps on taut complex manifolds that we need later.
Most of the material presented here is taken from Chapter~2.1 and Chapter~2.4 in 
Abate \cite{Abate:IterationTheory89} (also, see Kobayashi \cite{Kobayashi:Hypcomplex98}). We begin with stating
a result that is due to Wu \cite[Theorem~C]{Wu:Normfamily67}.

\begin{result}\label{Res:CCT}
Let $\mfd{X}$ be a taut complex manifold and let $f\in\hol(\mfd{X},\,\mfd{X})$ be such that $f(z_0)=z_0$ for some $z_0\in \mfd{X}$. 
Then:
\begin{itemize}
\item[$a)$] the spectrum of the derivative of $f$ at $z_0$, $f'(z_0)$, is contained in $\overline{\D}$. 
\smallskip

\item[$b)$] $f'(z_0)=\I$ if and only if $f$ is the identity function.
\smallskip

\item[$c)$] The tangent space $T_{z_0}\mfd{X}$ admits a $f'(z_0)$-invariant splitting $T_{z_0}\mfd{X}=L_{N}\oplus L_U$ such that the spectrum of 
$f'(z_0)\big{|}_{L_N}$ is contained in $\D$, the spectrum of $f'(z_0)\big{|}_{L_U}$
is contained in $\partial\,\D$ and $f'(z_0)\big{|}_{L_U}$ is diagonalizable.
\end{itemize}
\end{result}
\noindent The subspace $L_{U}$ of $T_{z_0}\mfd{X}$ is called the {\em unitary space} of $f$ at the fixed point $z_0$ and the subspace $L_N$ is called
the {\em nilpotent space} of $f$ at $z_0$. 
\smallskip

Before we state the next result, we need a definition. Let $\mfd{X}$ be a complex manifold.
A {\em holomorphic retraction} of $\mfd{X}$ is a holomorphic map $\rho:\mfd{X}\lrarw\mfd{X}$
such that $\rho^2 = \rho$. A {\em holomorphic retract} of $\mfd{X}$ is the image of $\mfd{X}$ under a
holomorphic retraction. It is known that any holomorphic retract of $\mfd{X}$ is a closed complex submanifold
of $\mfd{X}$. We now state
\begin{result}\label{Res:existretract}
Let $\mfd{X}$ be a taut complex manifold and $f\in\hol(\mfd{X},\,\mfd{X})$. Assume that the sequence
$\{f^k\}$ of iterates of $f$ is not compactly divergent. Then there exist a complex submanifold $\mfd{M}$
of $\mfd{X}$ and a holomorphic retraction $\rho:\mfd{X}\lrarw\mfd{M}$ such that every limit point
$h\in\hol(\mfd{X},\,\mfd{X})$ of $\{f^k\}$ is of the form
\[
h =\gamma\circ\rho,
\]
where $\gamma$ is an automorphism of $\mfd{M}$. Moreover, even $\rho$ is a limit point of 
the sequence $\{f^k\}$.
\end{result}
The above result is due to Abate \cite{Abate:ISGTM88}. The manifold $\mfd{M}$ above is called the
{\em limit manifold} of $f$ and its dimension is called the 
{\em limit multiplicity} of $f$. 
If $f\in\hol(\mfd{X},\,\mfd{X})$ be such that $f(z_0)=z_0$ for some $z_0\in\mfd{X}$, $\mfd{X}$ being taut,
then following is an easy consequence of the above results:

\begin{corollary}\label{Cor:unitspacetangspaceoflimitmani}
Given $f\in\hol(\mfd{X},\,\mfd{X})$ with $f(z_0)=z_0$ for some $z_0\in\mfd{X}$ and $\mfd{X}$ being a
 taut complex manifold, the unitary space of $f$ at $z_0$ is the tangent 
space at $z_0$ of the limit manifold of $f$. In particular, the limit multiplicity of $f$ is the number of
 eigenvalues of $f'(z_0)$ that belong to $\partial\,\D$ counted with multiplicity.
\end{corollary}

We also need a result due to Abate \cite{Abate:ISGTM88} that gives a characterization
for the sequence $\{f^k\}\in\hol(\mfd{X},\,\mfd{X})$ to be convergent.
\begin{result}\label{Res:iterconv}
Let $\mfd{X}$ be a taut complex manifold and let $f\in\hol(\mfd{X},\,\mfd{X})$. Then the sequence
of iterates $\{f^k\}$ converges in $\hol(\mfd{X},\,\mfd{X})$ if and only if $f$ has a fixed point $z_0\in\mfd{X}$ such
that the spectrum of $f'(z_0)$ is contained in $\D\cup\{1\}$.
\end{result}
We end this section with a result due to Vigu\'{e} about the fixed point set of a holomorphic self-map. 
Given $f\in\hol(\mfd{X},\,\mfd{X})$, we shall denote by ${\rm{Fix}}(f)$ the set of fixed points of $f$. 
\begin{result}[Vigu\'{e},\,\cite{Vigue:FPlhm}]\label{Res:Vig}
Let $\mfd{X}$ be a taut complex manifold, $f\in\hol(\mfd{X},\,\mfd{X})$. Then ${\rm{Fix}}(f)$
 is a closed complex submanifold of $\mfd{X}$.
Moreover, for $x\in {\rm{Fix}}(f)$, we have
\[
 T_{x}({\rm{Fix}}(f))\,=\,\{\xi\in T_{x}\mfd{X}\,:\,f'(x)\xi=\xi\}.
 \]
\end{result}
\noindent{Also, see \cite[Theorem~5.5.8]{Kobayashi:Hypcomplex98} for details.}

\section{Two preliminary lemmas}\label{S:prelim_results}
In this section, we state two closely related lemmas. Lemma~\ref{L:indmapsymprod} is one of the key tools in the proof of
the two main results of this paper. Both lemmas are simple once we appeal to a result by Zwonek. We begin by stating this result. 

\begin{result}[Zwonek, \cite{wz:ftpspcm20}]\label{Res:kobhypsymprod}
Let $\Omega\subset\C$ be a domain and let $n\in\nat$, $n\geq 2$, be fixed. If $\#(\C\setminus\Omega)\geq 2n$ then 
$\Sigma^n(\Omega)$ is Kobayashi complete. If $\#(\C\setminus \Omega)< 2n$ then $\Sigma^n(\Omega)$ 
contains a non-constant holomorphic image of $\C$ and thus $\Sigma^n(\Omega)$ is not Kobayashi hyperbolic. 
\end{result}

\begin{lemma}\label{L:prespectra}
Consider a domain $\Omega\subset\C$ and $n\in\nat$, $n\geq 2$, such that $\#(\C\setminus\Omega)\geq 2n$.
Let $\Psi$ be a holomorphic self-map of $\Sn$. Then for every 
$W_1,\,W_2\in\Sn$ such that $\ch(W_1)=\ch(W_2)$, we have $\ch(\Psi(W_1))=\ch(\Psi(W_2))$, where 
$\ch : \Sn\lrarw\Sigma^n(\Omega)$ is as defined in Section~\ref{S:intro}.
\end{lemma}
\begin{proof}
Fix $W_1,\,W_2\in\Sn$ such that $\ch(W_1)=\ch(W_2)$. We know that there exists $C\in M_n(\C)$ and a
strictly upper triangular matrix $U$ such that
%\[
$W_1=\exp(-C)\,(D+U)\,\exp(C)$,
%\]
where $D$ is a diagonal matrix such that $\ch(D)=\ch(W_1)$. Now consider the map $f:\C\lrarw M_n(\C)$ defined by 
\[
f(\zt):=\exp(-C\,\zt)\,(D+\zt\,U)\,\exp(C\,\zt) \ \ \text{for all $\zt\in\C$}.
\]
Note that $\ch(f(\zt))=\ch(D+\zt\,U)=\ch(D)$, hence $f(\C)\subset\Sn$. This allows us to define the map
$g(\zt):=\ch\circ \Psi\circ f(\zt)$ for all $\zt\in\C$. Note that $g$ is a holomorphic map from 
$\C$ to $\Sigma^n(\Omega)$.
Since $\#(\C\setminus\Omega)\geq 2n$, by Result~\ref{Res:kobhypsymprod}, it follows 
that $\Sigma^n(\Omega)$ is 
Kobayashi hyperbolic. Then using Result~\ref{Res:LT} we get that $g$ is a constant function.
Hence
\[
\ch(\Psi(D))=g(0)=g(1)=\ch(\Psi(W_1)).
\]
Proceeding similarly we get $\ch(\Psi(D))=\ch(\Psi(W_2))$ whence $\ch(\Psi(W_1))=\ch(\Psi(W_2))$.
Since the choice of $W_1,\,W_2\in\Sn$ (satisfying $\ch(W_1)=\ch(W_2)$) was arbitrary, the lemma follows.
\end{proof}

The above lemma is motivated from that of \cite[Theorem~1]{Ransford-White: holselfspectunitball91}
by Ransford--White. It also appeared in \cite{VC:CnonhomMatdom} but we present the proof here for 
completeness. Further, with the help of this lemma, we prove the following result:

\begin{lemma}\label{L:indmapsymprod}
Consider a domain $\Omega\subset\C$ and $n\in\nat$, $n\geq 2$, such that $\#(\C\setminus\Omega)\geq 2n$.
Let $\Psi$ be a holomorphic self-map of $\Sn$. Then there exists a unique holomorphic 
self-map $G_{\Psi}$ of $\sym{n}$ such that $G_{\Psi}\circ\ch=\ch\circ\Psi$ (also see Figure~\ref{F:comd1}). 
\end{lemma}
\begin{proof}
Consider a relation $G_\Psi$ from $\Sigma^n(\Omega)$ into $\Sigma^n(\Omega)$ defined by
\begin{equation*}\label{E:autsym}
G_\Psi(z):=\ch\circ \Psi\circ{\ch}^{-1}(z) \ \ \ \forall z\in \Sigma^n(\Omega).
\end{equation*}
From Lemma~\ref{L:prespectra}, it follows that for each $z\in \Sigma^n(\Omega)$, $G_\Psi(z)$ is a singleton.
Hence $G_\Psi:\Sigma^n(\Omega)\lrarw\Sigma^n(\Omega)$ is a well-defined map that satisfies the relation 
$G_{\Psi}\circ\ch=\ch\circ\Psi$. The lemma now follows once we prove the following claim:
\smallskip

\noindent{\bf Claim.} $G_\Psi$ is holomorphic.
\smallskip

\noindent To see this, fix $z\in\Sigma^n(\Omega)$. Now consider the polynomial
$P_z(t):=t^n+\sum_{j=1}^n(-1)^jz_j\,t^{n-j}$ and define the map $\kappa:\Sigma^n(\Omega)\lrarw M_n(\C)$ by setting
\begin{equation*}\label{E:tau}
\kappa(z):=\mathsf{C}\big(P_z\big),
\end{equation*}
where $\mathsf{C}\big(P_z\big)$ denotes the companion matrix of the polynomial $P_z$. Recall,
given a monic polynomial of degree $k$ of the form $p(t)=t^k+\sum_{j=1}^ka_j\,t^{k-j}$, where $a_j\in\C$,
the \emph{companion matrix} of $p$ is the matrix $\mathsf{C}(p)\in M_{k}(\C)$ given by
\[
\mathsf{C}(p):=
\begin{bmatrix}
\ 0  & {} & {} & -a_k \ \\
\ 1  & 0  & {} & -a_{k-1} \ \\
\ {} & \ddots  & \ddots & \vdots \ \\
\ \text{\LARGE{0}} &   & 1 & -a_{1} \
\end{bmatrix}_{k\times k}.
\]
It is a fact that $\ch(\mathsf{C}(p))=(a_1,\dots,a_k)$. From this, it follows that 
$\kappa$ is holomorphic and $\ch\circ\kappa=\I$ on
$\Sigma^n(\Omega)$.
This, in particular, implies that $\kappa(z)\in\ch^{-1}(z)$. Applying Lemma~\ref{L:prespectra} again, we see that
$\ch\circ \Psi\circ{\ch}^{-1}(z)=\ch\circ \Psi\circ\kappa(z)$, i.e., $G_\Psi(z)=\ch\circ \Psi\circ\kappa(z)$.
Since each of the maps $\ch, \Psi, \kappa$ are holomorphic, the 
claim follows.
\end{proof}

\section{Preparations for the proof of Theorem~\ref{T:mainT1}}\label{S:Prep}
In this section, we devise certain ingredients that play a crucial role in the proof of Theorem~\ref{T:mainT1}. We
first show that given a point $A\in M_n(\C)$ ($\equiv\C^{n^2}$) there is a polydisc centered at $A$ on which the map
$\ch$ could be decomposed. We also describe the utility of this decomposition to our proof of Theorem~\ref{T:mainT1}.
In what follows, given integers $j< k$, $\intgR{j}{k}$
will denote the set of integers $\{j, j+1,\dots, k\}$.

\subsection{Local decomposition of $\ch$}\label{SS:locdecochi}
Recall that given $x\in\C^n$, $P_x(t)$ is the polynomial $t^n+\sum_{j=1}^n(-1)^jx_j\,t^{n-j}$. Given $n\geq 2$, suppose there exist positive 
integers $n_i,\,i\in\intgR{1}{m}$ such that $\sum_{i=1}^m n_i=n$.
Consider the map $\tau:\prod_{i=1}^m\sym{n_i}\lrarw\sym{n}$ defined by
\begin{equation}\label{E:deftau}
\tau(x_1,\dots,x_m)=y, \ \  \text{where $y$ satisfies} \ \
P_{y}(t)=\prod_{i=1}^m P_{x_i}(t).
\end{equation}
Note that $\tau$ is a holomorphic surjective map. We now state the result regarding the local decomposition of $\ch$.
\begin{lemma}\label{L:localdecom}
Let $A\in\Sn$, $n\geq 2$, and write $\sigma(A):=\{\lam_1,\dots,\lam_m\}$ such that for each $i\in\intgR{1}{m}$, $n_i$ is the algebraic multiplicity
of $\lam_i$. Then there exists a $\delta>0$ such that on the polydisc ${\sf{P}}(A;\,\delta)$, the map 
$\ch$ decomposes as $\ch\,=\,\tau\circ\theta$, where $\theta:{\sf{P}}(A;\,\delta)\lrarw\prod_{i=1}^m\sym{n_i}$ is a holomorphic open map and
$\tau:\prod_{i=1}^m\sym{n_i}\lrarw\sym{n}$, as defined above, is a biholomorphism from $\theta({\sf{P}}(A;\,\delta))$ onto $\ch({\sf{P}}(A;\,\delta))$.
\end{lemma}
\begin{proof}
Choose an $r>0$ such that $r<\min\{{|\lam_i-\lam_j|}\big{/}2\,:\,i,j\in\intgR{1}{m},\,i\neq j\}$ and
the discs $\mathbb D(\lam_i;\,r):=\{\zt\in\C\,:\,|\zt-\lam_i|<r\}$ are contained in $\Omega$.
Now using the continuity of the map $\ch$ and the fact that the roots of a polynomial\,---\,as a function of its coefficients\,---\,vary continuously, we 
can find a $\delta>0$ such that for any $W \in{\sf{P}}(A;\,\delta)\subset\C^{n^2}$ 
the number of eigenvalues of $W$ in $\mathbb D(\lam_i;\,r)$, counted with multiplicity, is $n_i$ for all $i\in\intgR{1}{m}$.
Given $W\in{\sf{P}}(A;\,\delta)$, denote by $\bv{\sigma_i(W)}$ a list of eigenvalues 
of $\sigma(W)$ that lie in the disc $\mathbb D(\lam_i;\,r)$ and are repeated with their multiplicity. Note that the number of elements in 
$\bv{\sigma_i(W)}$ is $n_i$ for each $i$. Now, we define the map $\theta:{\sf{P}}(A;\,\delta)\lrarw\prod_{i=1}^m\sym{n_i}$ by 
\begin{equation}\label{E:deftheta}
\theta(W)=\big(\theta_1(W),\dots,\theta_m(W)\big), \ \text{where each $\theta_i(W)$ satisfy} \ 
P_{\theta_i(W)}(t)=\prod_{\mu\in\bv{\sigma_i(W)}}(t-\mu). 
\end{equation}
It is not difficult to see that $\ch$ and $\theta$ are open maps\,---\,see Section~\ref{S:propchi} for details.
It follows from \eqref{E:deftau} and \eqref{E:deftheta} that 
\begin{equation}\label{E:locdeco}
\tau(\theta_1(W),\dots,\theta_m(W))\,=\,\ch(W) \ \ \text{for all $W\in {\sf{P}}(A;\,\delta)$}.
\end{equation}

We now show that $\tau$ is a biholomorphism from the open set $\theta({\sf{P}}(A;\,\delta))$ onto the open set $\ch({\sf{P}}(A;\,\delta))$.
Notice we only need to show that $\tau$ is injective on $\theta({\sf{P}}(A;\,\delta))$. Suppose
\begin{equation}\label{E:injassp}
\tau(x_1,\dots,x_m)=\tau(y_1,\dots,y_m),
\end{equation}
where $(x_1,\dots,x_m)=\theta(W_1)$ and $(y_1,\dots,y_m)=\theta(W_2)$ for some $W_1, W_2\in {\sf{P}}(A;\,\delta)$. 
It follows from the definition of $\tau$ and \eqref{E:injassp} that 
\begin{equation}\label{E:eqprodpoly}
\prod_{j=1}^{m} P_{x_j}(t)=\prod_{j=1}^{m} P_{y_j}(t).
\end{equation}
Fix a $j\in\intgR{1}{m}$. Since $x_j=\theta_j(W_1)$ and $W_1\in{\sf{P}}(A;\,\delta)$, the zeros of $P_{x_j}(t)$ lie 
in the disc $\mathbb D(\lam_j;\,r)$.
Similarly for any $k\neq j$\,---\,since $y_k=\theta_k(W_2)$ and $W_2\in{\sf{P}}(A;\,\delta)$\,---\,the zeros of 
the polynomial
$P_{y_k}(t)$ lie in the disc $\mathbb D(\lam_k;\,r)$. As $\mathbb D(\lam_j;\,r)\cap \mathbb D(\lam_k;\,r)=\emptyset$
for $k\neq j$ whence the zeros of the polynomial $P_{x_j}(t)$ are also the zeros of $P_{y_j}(t)$. Reversing 
this argument we see that the zeros of $P_{x_j}(t)$ and $P_{y_j}(t)$ coincide. Hence $x_j=y_j$ for each $j\in\intgR{1}{m}$
showing the injectivity of $\tau$ on $\theta({\sf{P}}(A;\,\delta))$.
\smallskip 

The holomorphicity of the map $\theta$ on ${\sf{P}}(A;\,\delta)$ now follows from \eqref{E:locdeco} together with the fact that
$\tau$ is a biholomorphism from $\theta({\sf{P}}(A;\,\delta))$ onto $\ch({\sf{P}}(A;\,\delta))$.
\end{proof}

We now present a lemma that paves the way towards the proof of Theorem~\ref{T:mainT1}.
\begin{lemma}\label{L:princcommdgm}
Given $n\in\nat,\,n\geq 2$, and a domain $\OM$ in $\C$ satisfying $\#(\C\setminus\OM)\geq 2n$. Let 
$\Psi\in\hol(\Sn,\,\Sn)$ such that $\Psi(A)=A$, $\Psi'(A)=\I$ for some $A\in\Sn$. Then there exist neighbourhoods 
$\mathcal{V}_j\subset{\sf{P}}(A;\,\delta)$, $j=1,2$, of $A$ satisfying $\mathcal{V}_2=\Psi(\mathcal{V}_1)$ such that if we define
$F_{\Psi}:\theta(\mathcal{V}_1)\lrarw\theta(\mathcal{V}_2)$ by
\begin{equation}
F_{\Psi}\equiv \big(\tau\big{|}_{\theta(\mathcal{V}_2)}\big)^{-1}\circ G_{\Psi}\circ\tau,
\end{equation}
where $\theta,\,\tau,\,{\sf{P}}(A;\,\delta)$ are as in Lemma~\ref{L:localdecom} and $G_{\Psi}$ is as in Lemma~\ref{L:indmapsymprod} then 
the following diagram is 
commutative: \vspace{-0.5cm}
\begin{figure}[h!]
\[
\begin{tikzcd}
 \mathcal{V}_1 \arrow [rr,"\Psi"]  \arrow[d, "\theta"]  & &\mathcal{V}_2  \arrow[d, "\theta"] \\
  \theta(\mathcal{V}_1) \arrow[rr, "F_{\Psi}"] \arrow[d, "\tau"]  & & \theta(\mathcal{V}_2)\arrow[d, "\tau"] \\
  \ch(\mathcal{V}_1) \arrow[rr,"G_\Psi"]& & \ch(\mathcal{V}_2)\\
   \end{tikzcd}
   \] \vspace{-1.0cm}
   \caption{}
   \label{F:comd2}
\end{figure}
\end{lemma}

\begin{proof} Since $\Psi(A)=A$ and $\Psi'(A)=\I$, the inverse function theorem implies that there are neighbourhoods
$\mathcal{V}_1,\,\mathcal{V}_2$ of $A$ that are contained in the polydisc ${\sf{P}}(A;\,\delta)$ such that $\Psi(\mathcal{V}_1)=\mathcal{V}_2$. 
Observe that we only need to show that on $\mathcal{V}_1$,
\[
 F_{\Psi}\circ\theta=\theta\circ\Psi.
 \]
Notice, by the definition of $F_{\Psi}$, we have
\[
 F_{\Psi}\circ\theta=\big[\big(\tau\big{|}_{\theta(\mathcal{V}_2)}\big)^{-1}\circ G_{\Psi}\circ\tau\big]\circ\theta
                             =\big(\tau\big{|}_{\theta(\mathcal{V}_2)}\big)^{-1}\circ G_{\Psi}\circ\ch.
                             \]
In the above we have used the identity $\tau\circ\theta=\ch$ on ${\sf{P}}(A;\,\delta)$. Now, since $G_{\Psi}\circ\ch=\ch\circ\Psi$,
the above equation becomes
\[
  F_{\Psi}\circ\theta=\big(\tau\big{|}_{\theta(\mathcal{V}_2)}\big)^{-1}\circ\ch\circ\Psi.
  \]
  Now, on $\mathcal{V}_2$, we have $\theta=\big(\tau\big{|}_{\theta(\mathcal{V}_2)}\big)^{-1}\circ\ch$. Putting this into the above equation
  gives us the desired equality. 
 \end{proof}

The following two points encapsulates the importance of Lemma~\ref{L:princcommdgm} and the commutative diagram therein
in our proof of Theorem~\ref{T:mainT1}:

\begin{itemize}[leftmargin=14pt]
\item The main goal in our proof of Theorem~\ref{T:mainT1}\,---\,when $A$ is diagonalizable\,---\,will be to prove that the trace of
$G'_{\Psi}$ at $a=\ch(A)$ is $n$. Since the map $G_{\Psi}$ and $F_{\Psi}$ are locally
biholomorphic conjugates of each other, it is sufficient to show that there is a basis $\mathfrak{B}$ of $\C^n$ such that the trace of
$F'_{\Psi}$ at $a^*=\theta(A)$ with respect to 
$\mathfrak{B}$ is $n$. 
\smallskip

\item The commutativity of the upper-half part of the above diagram enables us in computing the trace of $F'_{\Psi}$ at $a^*$
with respect to an appropriately chosen basis $\mathfrak{B}$ as mentioned above. In fact, in the next subsection we shall construct a basis $\mathfrak{B}$ and
derive a very important information regarding the trace 
of certain diagonal blocks of $[F'_{\Psi}(a^*)]_{\mathfrak{B}}$:=the derivative matrix of $F_{\Psi}$ at $a^*$ with respect to $\mathfrak{B}$. 
\end{itemize}

\subsection{An important proposition}\label{SS:crucprop}
We continue with the set-up as in Lemma~\ref{L:princcommdgm}. Assuming that
the matrix $A$ in the aforementioned lemma is a diagonal matrix, we derive an important information
regarding the trace of $[F'_{\Psi}(a^*)]_{\mathfrak{B}}$ with respect to an appropriately chosen basis 
$\mathfrak{B}$. For simplicity, we shall write the maps $F_{\Psi},\,G_{\Psi}$ as $F,\,G$ respectively. Note the map $F:\theta(\mathcal{V}_1)\longrightarrow\theta(\mathcal{V}_2)$, where $\mathcal{V}_1$ and $\mathcal{V}_2$ are as in Lemma~\ref{L:princcommdgm}, can be written as 
$F=(F_1,\dots,F_m) \ \text{such that} \ F_i(\theta(\mathcal{V}_1))\subset\Sigma^i(\Omega)$ for all $i\in\intgR{1}{m}$. Also, if we let $\mathfrak{A}_i:=\{{\bf e}^i_1,\dots,{\bf e}^i_{n_i}\}$ denote the standard basis of $\C^{n_i}$ then write
$F_i:=\sum_{j=1}^{n_i}F_{i,j}\,{\bf e}^i_j$ on $\theta(\mathcal{V}_1)$.
The following result by Sun is at the heart of the proof of the main result of this subsection:

\begin{result}[paraphrasing of Corollary~1.2 in \cite{Sun:Varspectnormal96}]\label{Res:sun}
Let $X\in M_n(\C)$ be a normal matrix with $\bv{\sigma(X)}=\{\zt_1,\dots,\zt_n\}$. Here $\bv{\sigma(X)}$
denotes a list of eigenvalues of $X$ repeated according to their multiplicity. Let $Y$
be any other matrix with $\bv{\sigma({Y})}=\{{\xi}_1,\dots,{\xi}_n\}$. Then there exists a permutation $\pi$ of $\intgR{1}{n}$ such that
\[
\max\{|\,{\xi}_{\pi(j)}-{\zt}_j\,|\,:\,j\in\intgR{1}{n}\}\,\leq\,n\,||\,{X}-Y\,||_{\rm{op}},
\]
where $||\bcdot||_{\rm{op}}$ denotes the operator norm of a matrix considered as a bounded linear operator on the Hilbert space $(\C^n,\,||\bcdot||_2)$. 
\end{result}

We are now in a position to state our main result of this subsection.
(In what follows, given a finite set of ordered vectors $\mathfrak{S}$ of $\C^k$, 
we shall denote by $[\mathfrak{S}]$ the matrix of $\mathfrak{S}$ with respect to the standard basis of $\C^k$).
\begin{proposition}\label{P:pincpropdiag}
Suppose the matrix $A$ in Lemma~\ref{L:princcommdgm} is a diagonal matrix with the $i_0$-th eigenvalue $\lam_{i_0}=0$. Then 
there is a basis $\mathfrak{B}_i, \ i\in\intgR{1}{m}$ of $\C^{n_i}$ such that the basis $\mathfrak{B}$ of $\C^n$ defined by 
\[
[\mathfrak{B}]=[\mathfrak{B}_1]\oplus\dots\oplus[\mathfrak{B}_m],
\]
has the property that the $i_0$-th diagonal block of size $n_{i_0}\times n_{i_0}$ of the matrix $\big[F'(a^*)\big]_{\mathfrak{B}}$ has trace $n_{i_0}$.
\end{proposition}
\begin{proof}
We first construct the basis $\mathfrak{B}_i:=\big\{{\bf v}^i_j\in\C^{n_i}\,:\,j\in\intgR{1}{n_i}\big\}$ for each $i\in\intgR{1}{m}$. These are defined by the equation
\[
P_{\theta_i(A)+\eta\,{\bf v}^i_j}(t)=(t-\lam_i)^{n_i}+\eta\,(t-\lam_i)^{n_i-j} \ \ \ \forall j\in\intgR{1}{n_i},
\]
where $P$ and $\theta_i$'s are as in Subsection~\ref{SS:locdecochi}. It is easy to see that
\[
{\bf v}^i_j\,:=\,
\begin{cases}
\big(0,\dots,0,1,\pi_{n_i-j}(\lam_i,\dots,\lam_i)\big), &\text{if $j\in\intgR{1}{n_i-1}$},\\
(0,\dots,0,1), &\text{otherwise},
\end{cases}
\]
where $\pi_{n_i-j}:\C^{n_i-j}\longrightarrow\C^{n_i-j}$ is the symmetrization map. Notice $\mathfrak{B}_i$ is a set of $n_i$ linearly independent vectors
and hence $\mathfrak{B}_i$ forms a basis of $\C^{n_i}$. Since $\lam_{i_0}=0$, we also see that $\mathfrak{B}_{i_0}= \mathfrak{A}_{i_0}$.
For each $i\in\intgR{1}{m}$, we write
\[
\theta_i=\sum_{j=1}^{n_i}\widehat{\theta}_{i,j}\,{\bf v}^i_j \ \ \ \text{and} \ \ \ F_i=\sum_{j=1}^{n_i}\widehat{F}_{i,j}\,{\bf v}^i_j.
\]
Since $\mathfrak{B}_{i_0}= \mathfrak{A}_{i_0}$, we have $\widehat{\theta}_{i_0,j}=\theta_{i_0,j}$ and
$\widehat{F}_{i_0,j}=F_{i_0,j}$ for all $j\in\intgR{1}{n_{i_0}}$. Denote by ${\bf V}^{i_0}_j:=
\big({\bf V}^{i_0}_{j,\,1},\dots,{\bf V}^{i_0}_{j,\,i},\dots,{\bf V}^{i_0}_{j,\,m}\big)\in\prod_{i=1}^{m}\C^{n_i}$
such that
${\bf V}^{i_0}_{j,\,i}=\bf{0}$ when
$i\neq i_0$ and ${\bf V}^{i_0}_{j,\,i_0}={\bf v}^{i_0}_j$.
\smallskip

{\noindent{\textbf{Claim.}}
\[
\Bigg[\intf{}{}{\partial \widehat{F}_{i_0,j}}{\partial{\bf V}^{i_0}_k}(a^*)\Bigg]=\Bigg[\intf{}{}{\partial F_{i_0,j}}{\partial {\bf E}^{i_0}_k}(a^*)\Bigg]=
 \mathbb{I}_{n_{i_0}}+{N},  \vspace{0.3cm}
 \]
where ${N}$ is an upper triangular nilpotent matrix and $\mathbb{I}_{n_{i_0}}$ is the identity
matrix of order $n_{i_0}$. Also, ${\bf E}^{i_0}_k\in\prod_{i=1}^{m}\C^{n_i}$ is a vector whose $i_0$-th component is ${\bf e}^{i_0}_k$ and every other component is the zero vector.}
\smallskip

\noindent{To establish the claim, we begin with the observation
\[
\intf{}{}{\partial F_{i_0,j}}{\partial{\bf E}^{i_0}_k}(a^*)=\lim_{\eps\to 0}\intf{}{}{F_{i_0,j}(a^*+\eps\,{\bf E}^{i_0}_k)-F_{i_0,j}(a^{*})}{\eps}.
\]
Now, by Lemma~\ref{L:princcommdgm}, $F_{i_0,j}=\theta_{i_0,j}\circ\Psi\circ\theta^{-1}$.
Substituting this together with the observation that
$\theta_{i_0,j}(A)=0$ gives us}
\begin{align}
\intf{}{}{\partial F_{i_0,j}}{\partial{\bf E}^{i_0}_k}(a^*)=\lim_{\eps\to 0}\intf{}{}{\theta_{i_0,j}\circ\Psi\circ\theta^{-1}(a^*+\eps\,{\bf E}^{i_0}_k)}{\eps}.
\label{E:upcumm}
\end{align}
Let us now write $\mathbb{I}_n=\mathbb{I}_{n_1}\oplus\dots\oplus\mathbb{I}_{n_m}$. Let $D^{i_0}_k\in M_{n_{i_0}}(\C)$ be the diagonal matrix
defined by
\[
D^{i_0}_k={\rm{diag}}[\omega_1,\omega_2,\dots,\omega_k,0,\dots,0], \ \ \ k\in\intgR{1}{n_{i_0}},
\]
where $\omega_j$'s are the roots of the equation $x^k+1=0$. Consider the matrix $D_k=\oplus_{i=1}^m\,W_i$, where $W_i=0\in M_{n_i}(\C)$, if $i\neq
i_0$ and $W_{i_0}=D^{i_0}_k$. Observe that (when $\eps$ is sufficiently small)
\[
\theta_i(A+{\eps}^{1/k}D_k)\,:=\,
\begin{cases}
\theta_i(A), &\text{if $i\neq i_0$},\\
\eps\,{\bf e}^{i_0}_k, &\text{otherwise}.
\end{cases}
\]
Hence $\theta(A+{\eps}^{1/k}D_k)=a^*+\eps\,{\bf E}^{i_0}_k$ for all $k\in\intgR{1}{n_{i_0}}$.
Substituting this into \eqref{E:upcumm} we get
\begin{align}\label{E:simplified}
\intf{}{}{\partial F_{i_0,j}}{\partial{\bf E}^{i_0}_k}(a^*)=
\lim_{\eps\to 0}\intf{}{}{\theta_{i_0,j}\circ\Psi(A+{\eps}^{1/k}D_k)}{\eps}
=\lim_{s\to 0}\intf{}{}{\theta_{i_0,j}\circ\Psi(A+s\,D_k)}{s^k}.
\end{align}
Since $\Psi(A)=A$ and ${\Psi}'(A)=\mathbb{I}$, for small enough $s$ we can write
\[
\Psi(A+s\,D_k)=A+s\,D_k+\sum_{j\geq2}\,B_js^j,
\]
where $B_j\in M_n(\C),\, j\geq 2$.
This, in particular, implies that
\[
||\Psi(A+sD_k)-(A+sD_k)||_{\rm{op}}=s^2\,M(s),
\]
where $M(s)$ is a continuous function in a neighbourhood of $0$. Now, when $s$ is sufficiently small both $\Psi(A+sD_k)$ and $(A+sD_k)$ lie
in $\mathcal{V}_2\subset {\sf{P}}(A;\,\delta)$ and $\mathcal{V}_1\subset {\sf{P}}(A;\,\delta)$ respectively. Furthermore, the 
non-zero eigenvalues of $A+sD_k$ that lie in the disc $\mathbb{D}(\lam_{i_0};\,r)\equiv\mathbb{D}(0;\,r)$
are $s\omega_1,\dots,s\omega_k$. Notice that the matrices $A+sD_k$ are all diagonal matrices.
So if we denote by $\mu^{i_0}_j(s),\,j\in\intgR{1}{n_{i_0}}$ the eigenvalues of 
$\Psi(A+sD_k)$ that lie in the disc $\mathbb{D}(0;\,r)$ then by Result~\ref{Res:sun}, there exists $\zeta_j(s)\in\D$, $j\in\intgR{1}{n_{i_0}}$ such that
\[
\mu^{i_0}_j(s)\,=\,
\begin{cases}
s\,\omega_j+\zeta_j(s)\,n\,s^2M(s), &\text{if $j\in\intgR{1}{k}$},\vspace{0.2cm}\\
\zeta_j(s)\,n\,s^2\,M(s), &\text{$j\in\intgR{k+1}{n_{i_0}}$}.
\end{cases}
\]
For a fixed $j \ge 1$, let $\mathscr{I}_j$ be the collection of all possible subsets of $\{1,2,\hdots,n_{i_0}\}$ of cardinality $j$ and let 
\[
 \mu_{ I}(s)=\mu_{i_1}(s)\mu_{i_2}(s)\hdots\mu_{i_j}(s),\ \text{where} \ { I}=\{i_1,i_2,\hdots,i_j\} \in\mathscr{I}_j.
 \] 
Then by definition
\[ \theta_{i_0,j}\circ \Psi (A+sD_k)=\sum_{{ I} \in{\mathscr{I}_j}} \mu_{ I}(s).\]
Now note that for $j>k$, $\mu_{ I}(s)=s^{j+1}h_{ I}(s)$, where $h_{ I}(s)$ are continuous functions in $s$ for every ${ I} \in \mathscr{I}_j.$ However,
for $j=k$, $\mu_{ I}(s)=s^j+s^{j+1}h_{ I}(s)$ only if ${ I}=\{1,2,\hdots,j\}$ and 
$\mu_{ I}(s)=s^{j+1} h_{ I}(s)$ otherwise. Thus we have
\[\intf{}{}{\partial F_{i_0,j}}{\partial{\bf E}^{i_0}_k}(a^*)= \lim_{s \to 0} \frac{\theta_{i_0,j}\circ \Psi (A+sD_k)}{s^k}=\begin{cases} 1 &j=k \\ 0 &j>k \end{cases}\]
which establishes the claim and consecutively proves our proposition.
\end{proof}

\subsection{Translation by a scalar matrix}\label{SS:implemma}
The purpose of this subsection is to devise a translation trick which is another main tool in computing the trace of the derivative of the map $F_{\Psi}$ as in Lemma~\ref{L:princcommdgm}. For this purpose, given
$\lam\in\C$, define the translation $L_{\lam}:M_n(\C)\lrarw M_n(\C)$ by $L_{\lam}(W):=W-\lam\mathbb{I}$. Notice $(L_{\lam})^{-1}=L_{-\lam}$. Furthermore, if
$\Omega\subseteq\C$ be any domain then $L_{\lam}(\Sn)=S_n(\Omega_\lam)$, where $\Omega_\lam=\{z-\lam:z\in\Omega\}$. Note that that if $\Omega$ satisfies the cardinality condition $\#(\C\setminus\OM)\geq 2n$ then so does $\OM_\lam$. Observe the map $L_{-\lam}$ introduces a map 
$G_{-\lam}:\Sigma^n(\Omega_\lam)\lrarw
\Sigma^n(\Omega)$ such that $G_{-\lam}\circ\ch=\ch\circ L_{-\lam}$, i.e., the following diagram commutes:
\vspace{-0.3cm}
\begin{figure}[h!]
\begin{equation*}
\begin{tikzcd}
 S_n(\Omega_\lambda) \arrow [rr,"L_{-\lambda}"]  \arrow[d, "\ch"]  & & S_n(\Omega)  \arrow[d, "\ch"] \\
  \Sigma^n(\Omega_\lambda)\arrow[rr, "G_{-\lambda}"]   & & \Sigma^n(\Omega) \\
   \end{tikzcd}
\end{equation*}
\vspace{-1.0cm}
\caption{}
\label{F:comd3}
\end{figure}\vspace{-0.2cm}

\noindent In fact, $G_{-\lam}\big(\pi_n(z_1,\dots,z_n)\big)=\pi_n(z_1+\lam,\dots,z_n+\lam)$, where $\pi_n:\C^n\longrightarrow\C^n$ is the
symmetrization map. Notice also that $G_{-\lam}$ is a biholomorphism, its
inverse $G_{\lam}$ is defined by $G_{\lam}\big(\pi_n(z_1,\dots,z_n)\big)=\pi_n(z_1-\lam,\dots,z_n-\lam)$.
\smallskip

We now define $\mathcal{V}_{j,\,\lam}=L_{\lam}(\mathcal{V}_j), j=1,2$, where $\mathcal{V}_1, \mathcal{V}_2$ are as in
Lemma~\ref{L:princcommdgm}. Then we have the following commutative diagram:
\vspace{-0.3cm}
\begin{figure}[h!]
\begin{equation*}
\begin{tikzcd}
 \mathcal{V}_{1,\,\lambda} \arrow [rr,"L_{-\lambda}"] \arrow[dd,swap,"\ch"] & &\mathcal{V}_1  \arrow[d, "\theta"]   \arrow[rr, "\Psi"]  & &\mathcal{V}_2  \arrow[d, "\theta"] \arrow[rr, "L_{\lambda}"] & & \mathcal{V}_{2,\,\lambda}\arrow[dd, swap,"\ch"]\\
  & & \theta(\mathcal{V}_1)\arrow[d, "\tau"] \arrow[rr, "F_{\Psi}"] & &\theta(\mathcal{V}_2)  \arrow[d, "\tau"]\\
  \ch(\mathcal{V}_{1,\,\lambda}) \arrow[rr,"G_{-\lambda}"]& & \ch(\mathcal{V}_1) \arrow[rr, "G_\Psi"] &  &\ch(\mathcal{V}_2)  \arrow[rr, "G_\lambda"] & &\ch(\mathcal{V}_{2,\,\lambda})\\
   \end{tikzcd}
\end{equation*}
\vspace{-1.0cm}\caption{}
\label{F:comd4}
\end{figure}
\vspace{-0.2cm}

Observe that $\mathcal{V}_{j,\,\lam}$ are neighbourhoods of $A_{\lam}:=A-\lam\mathbb{I}$ of the same type as $\mathcal{V}_j$'s are that of $A$. In other words, if we take
$B\in\mathcal{V}_{j,\,\lam}$ then the spectrum of $B$ is contained in the union of disjoint discs
$\mathbb{D}(\lam_i-\lam,\,r)$ that are centered at $\lam_i-\lam$ with radius $r$,
$i\in\intgR{1}{m}$, where $r>0$ is chosen as in the proof of Lemma~\ref{L:localdecom}. Moreover, the number of eigenvalues of $B$, counted with multiplicity, that lie in the disc $\mathbb{D}(\lam_i-\lam,\,r)$ is $n_i$,
$i\in\intgR{1}{m}$. Proceeding exactly as in the proof of Lemma~\ref{L:localdecom} and Lemma~\ref{L:princcommdgm}, we see that there are maps $\theta_{\lam},\tau_\lam$ and
$F_{-\lam},\,F_{\lam}$ such that the diagram in Figure~\ref{F:comd5} is commutative.
\vspace{-0.3cm}
\begin{figure}[h!]
\begin{equation*}
\begin{tikzcd}
 \mathcal{V}_{1,\,\lambda} \arrow [rr,"L_{-\lambda}"]  \arrow[d, "\theta_{\lambda}"]  & &\mathcal{V}_1  \arrow[d, "\theta"] \arrow[rr, "\Psi"]  & &\mathcal{V}_2  \arrow[d, "\theta"] \arrow[rr, "L_{\lambda}"] & & \mathcal{V}_{2,\,\lambda}\arrow[d,"\theta_{\lambda}"]\\
  \theta_\lambda(\mathcal{V}_{1,\,\lambda}) \arrow[rr, "F_{-\lambda}"] \arrow[d, "\tau_\lambda"]  & & \theta(\mathcal{V}_1)\arrow[d, "\tau"] \arrow[rr, "F_{\Psi}"] & &\theta(\mathcal{V}_2)  \arrow[d, "\tau"] \arrow[rr, "F_{\lambda}"] & & \theta_\lambda(\mathcal{V}_{2,\,\lambda})\arrow[d,"\tau_\lambda"]\\
  \ch(\mathcal{V}_{1,\,\lambda}) \arrow[rr,"G_{-\lambda}"]& & \ch(\mathcal{V}_1) \arrow[rr, "G_\Psi"] & &\ch(\mathcal{V}_2)  \arrow[rr, "G_\lambda"] & &\ch(\mathcal{V}_{2,\,\lambda})\\
   \end{tikzcd}
\end{equation*}
\vspace{-1.0cm}\caption{}
\label{F:comd5}
\end{figure}
\vspace{-0.2cm}

Define the maps $\Psi_\lam\,:=\,L_\lam\circ\Psi\circ L_{-\lam}$, $F_{\Psi_\lam}:=F_\lam\circ F_\Psi\circ F_{-\lam}$ and
$G_{\Psi_\lam}:=G_{\lam}\circ G_{\Psi}\circ G_{-\lam}$ and set $a^{*}_{\lam}:=\theta_\lam(A_\lam)$. 
The following lemma is the translation trick that we alluded to in the start of this subsection:
\smallskip

\begin{lemma}\label{L:diagblocksim}
Let $\mathfrak{B}_{\lam,i}$, $\mathfrak{B}_{0,i}$ be two bases of $\C^{n_i}$, $i\in\intgR{1}{m}$ and consider the bases $\mathfrak{B}_{\lam}$, 
$\mathfrak{B}_{0}$ of $\C^n$ that are defined by
\[
[\mathfrak{B}_{\lam}]=[\mathfrak{B}_{\lam,1}]\oplus\dots\oplus[\mathfrak{B}_{\lam,m}] \ \ \ \text{and} \ \ \
[\mathfrak{B}_{0}]=[\mathfrak{B}_{0,1}]\oplus\dots\oplus[\mathfrak{B}_{0,m}].
\]
Then the $i$-th diagonal block of size $n_i\times n_i$ of the matrices 
$[F'_{\Psi_\lam}(a^{*}_{\lam})]_{\mathfrak{B}_{\lam}}$ and $[F'_{\Psi}(a^{*})]_{\mathfrak{B}_{0}}$ are similar 
for every $i\in\intgR{1}{m}$.
\end{lemma}
\begin{proof} 
Since $F_{\Psi_\lam}=F_\lam\circ F_\Psi\circ F_{-\lam}$ on $\theta_{\lam}(\mathcal{V}_{1,\,\lam})$, by chain rule, we get\vspace{0.2cm}
\begin{equation}\label{E:flam&f-lam$fpsi}
\Big[F'_{\Psi_\lam}(a^{*}_{\lam})\Big]_{\mathfrak{B}_{\lam}}=
\Big[F'_\lam(a^{*}_{\lam})\Big]^{\mathfrak{B}_{\lam}}_{\mathfrak{B}_{0}}\cdot
\Big[F'_{\Psi}(a^{*})\Big]_{\mathfrak{B}_{0}}\cdot
\Big[F'_{-\lam}(a^{*}_{\lam})\Big]^{\mathfrak{B}_{0}}_{\mathfrak{B}_{\lam}}\vspace{0.2cm}.
\end{equation}
As noted in the beginning of this section $G_{\lam},\,G_{-\lam}$ are biholomorphisms and consequently, so are $F_{\lam}$, $F_{-\lam}$. Also, observe 
if we write $F_\lam=(F_{\lam,\,1},\dots,F_{\lam,\,m})$ and $F_{-\lam}=(F_{-\lam,\,1},\dots,F_{-\lam,\,m})$ then for every $x=(x_1,\dots,x_m)$ and $i\in
\intgR{1}{m}$, the values $F_{\lam,\,i}(x),\, F_{-\lam,\,i}(x)$ are independent of $x_j$, $j\neq i$. It follows from these two facts and from our choice of bases
$\mathfrak{B}_{\lam}$, $\mathfrak{B}_{0}$ that there exist invertible matrices $E_{\lam,\,i}$ of order $n_i$ such that
\begin{equation}\label{E:flaminv}
\Big[F'_\lam(a^{*}_{\lam})\Big]^{\mathfrak{B}_{\lam}}_{\mathfrak{B}_{0}}=E_{\lam,\,1}\oplus\dots\oplus E_{\lam,\,m}.
\end{equation}
Notice that since $F_{\lam}\circ F_{-\lam}\equiv \mathbb{I}$ on a small neighbourhood of $a^{*}_{\lam}$, we have
\begin{equation}\label{E:f-laminv}
\Big[F'_{-\lam}(a^{*}_{\lam})\Big]^{\mathfrak{B}_{0}}_{\mathfrak{B}_{\lam}}=
E^{-1}_{\lam,\,1}\oplus\dots\oplus E^{-1}_{\lam,\,m}.
\end{equation}
The formula for multiplying block matrices together with
\eqref{E:flam&f-lam$fpsi}, \eqref{E:flaminv} and \eqref{E:f-laminv} proves the result.
\end{proof}

\section{The proof of Theorem~\ref{T:mainT1}}\label{S:diagcase}
This section is devoted to the proofs of Theorem~\ref{T:mainT1} and Corollary~\ref{cor:locconj}. 
We use tools and techniques developed in 
Section~\ref{S:Prep} to prove Theorem~\ref{T:mainT1} when $A$ is diagonalizable. When $A$ is non-derogatory,
our proof of the theorem crucially depends on an explicit construction of a right inverse
of the map $\ch$ passing through the point $A$\,---\,which is independent of Section~\ref{S:Prep}.
But first we prove an important lemma.
\begin{lemma}\label{L:probconjinva}
Consider $\Omega\subset\C$ and $n\in\nat,\,n\geq 2$, such that $\#(\C\setminus\Omega)\geq 2n$. 
Let $A\in\Sn$ and $\Psi\in\hol(\Sn,\,\Sn)$ be such that $\Psi(A)=A$ and 
$\Psi'(A)=\mathbb{I}$. Let $B\in M_n(\C)$ be such that $B=S\,A\,S^{-1}$ for some invertible matrix
$S\in M_n(\C)$. Consider $\Phi\in\hol(\Sn,\,M_n(\C))$
defined by $\Phi={C_{S}}^{-1}\circ\Psi\circ {C_{S}}$, where $C_{S}(W)=S^{-1}\,W\,S$ for all $W\in M_n(\C)$. Then 
\begin{itemize}
\item[$1)$] $\Phi\in\hol(\Sn,\,\Sn)$ with $\Phi(B)=B$ and $\Phi'(B)=\mathbb{I}$.
\smallskip

\item[$2)$] Let $G_{\Phi}\in\hol(\Sigma^n(\Omega),\,\Sigma^n(\Omega))$ be the map associated to the map
$\Phi$ (see Lemma~\ref{L:indmapsymprod}). Then $G_{\Phi}\equiv G_{\Psi}$.
\end{itemize}
\end{lemma}

\begin{proof} The proof of Part~$(1)$ is straightforward. So let us prove Part~$(2)$. Choose a point 
$x\in \Sigma^n(\Omega)$ and fix it. 
Let $W\in\Sn$ be such that $\ch(W)=x$. Then using the identity $G_{\Phi}\circ\ch=\ch\circ\Phi$
we see that $G_{\Phi}(x)=\ch(\Phi(W))=
\ch({C_{S}}^{-1}\circ\Psi\circ {C_{S}}(W))$. Notice that $\ch\circ{C_{S}}^{-1}=\ch$ and hence we have
$G_{\Phi}(x)=\ch(\Psi\circ C_{S}(W))$. Now, we use the identity $G_{\Psi}\circ\ch=\ch\circ\Psi$, to get 
$\ch(\Psi\circ C_{S}(W))=G_{\Psi}(\ch(C_{S}(W))$.
Since $\ch\circ C_{S}=\ch$, we have $G_{\Phi}(x)=G_{\Psi}(\ch(W))=G_{\Psi}(x)$.
As $x$ is arbitrary, the conclusion follows.
\end{proof}

We are now ready for 

\begin{proof}[The proof of Theorem~\ref{T:mainT1}]
\noindent{\bf{When $A$ is diagonalizable:}}
Let $A\in\Sn$ be a diagonalizable matrix and $\Psi\in\hol(\Sn,\,\Sn)$ be such that $\Psi(A)=A$ and 
$\Psi'(A)=\mathbb{I}$. Because of Lemma~\ref{L:probconjinva}, we can assume, without loss of generality, that $A$ is a diagonal 
matrix with eigenvalues $\lam_i$ with algebraic multiplicity $n_i$, $i\in\intgR{1}{m}$. We shall now compute
the trace of $G'_{\Psi}$ at the point $a=\ch(A)$. Invoking Lemma~\ref{L:princcommdgm}, it is necessary and
sufficient to compute the trace of $F'_{\Psi}$ at $a^{*}:=\theta(A)$, where $F_{\Psi}$ and $\theta$ are as in 
Lemma~\ref{L:princcommdgm}.
\smallskip

Consider a basis $\mathfrak{B}$ of $\C^n$ of the form $[\mathfrak{B}]=[\mathfrak{B}_{1}]\oplus\dots\oplus[\mathfrak{B}_{m}]$, where 
$\mathfrak{B}_{i}$ is a basis of $\C^{n_i}$. Furthermore, for $\mathfrak{B}_i:=\{{\bf v}^i_j:i\in\intgR{1}{m},\, j\in\intgR{1}{n_i}\}$, let us
denote by ${\bf V}^{i}_j\in\prod_{k=1}^{m}\C^{n_k}$ the vector defined
by ${\bf V}^{i}_j:=
\big({\bf V}^{i}_{j,\,1},\dots,{\bf V}^{i}_{j,\,k},\dots,{\bf V}^{i}_{j,\,m}\big)\in\prod_{k=1}^{m}\C^{n_k}$
such that
${\bf V}^{i}_{j,\,k}=\bf{0}$ when
$k\neq i$ and ${\bf V}^{i}_{j,\,i}={\bf v}^{i}_j$.
If we represent $F_{\Psi}=(F^{1}_{\Psi},\dots,F^{m}_{\Psi})$, we can write
$F^{i}_{\Psi}=\sum_{j=1}^{n_i}F^{i}_{\Psi,\,j}{\bf v}^{i}_j$. Notice that the matrix of the derivative of $F_{\Psi}$ at $a^{*}$ with respect to the 
basis $\mathfrak{B}$ is
\[
\Big[F'_{\Psi}(a^{*})\Big]_{\mathfrak{B}}=\oplus_{i=1}^{m}\Bigg[\intf{}{}{\partial F^{i}_{\Psi,\,j}}{\partial{\bf V}^{i}_k}(a^*)\Bigg]
\ \ \ \text{and} \ \ \ {\rm{trace}}\Big[F'_{\Psi}(a^{*})\Big]_{\mathfrak{B}}=\sum_{i=1}^m{\rm{trace}}
\Bigg[\intf{}{}{\partial F^{i}_{\Psi,\,j}}{\partial{\bf V}^{i}_k}(a^*)\Bigg].
\]
Now, we prove the following claim:
\smallskip

\noindent{\textbf{Claim.}} The trace of $F'_{\Psi}$ at $a^{*}$ with respect to the basis $\mathfrak{B}$ is $n$.
\smallskip

\noindent{To see the claim, fix $i\in\intgR{1}{m}$. We now consider the map $\Psi_{\lam_i}$ and the associated maps $F_{\Psi_{\lam_i}}$ and $G_{\Psi_{\lam_i}}$ that are obtained by replacing $\lam$ with $\lam_i$ in the definition of maps $\Psi_{\lam},\,G_{\Psi_{\lam}},\,F_{\Psi_{\lam}}$ introduced just before 
Lemma~\ref{L:diagblocksim}. 
Notice that since $A$ is a diagonal matrix so is $A-\lam_i\I$. Moreover, $\sigma(A-\lam_i\I)=\{\mu_1,\dots,
\mu_m\}$ where $\mu_j=\lam_j-\lam_i$ for all $j\in\intgR{1}{m}$. In particular $\mu_i=0$. 
We now use Proposition~\ref{P:pincpropdiag} by taking $A-\lam_i\I$ as $A$, $\Psi_{\lam_i}$ as $\Psi$ 
and $F_{\Psi_{\lam_i}}$ as $F_{\Psi}$. Thus
there exists a basis $\mathfrak{D}_{i}$ of $\C^{n_i}$
such that with respect to the basis $\mathfrak{D}$ of $\C^n$ defined by $[\mathfrak{D}]\,=\,[\mathfrak{D}_{1}]\oplus\dots\oplus[\mathfrak{D}_{m}]$, the $i$-th diagonal block of order $n_i$ of the 
matrix $\big[F'_{\Psi_{\lam_i}}(a^*_{\lam_i})\big]_{\mathfrak{D}}$ has trace $n_{i}$. 
By Lemma~\ref{L:diagblocksim}, this latter block is similar to 
the $i$-th diagonal block of $\big[F'_{\Psi}(a^*)\big]_{\mathfrak{B}}$ and hence 
\[
 {\rm{trace}}\,\Bigg[\intf{}{}{\partial F^{i}_{\Psi,\,j}}{\partial{\bf V}^{i}_k}(a^*)\Bigg]=n_i.
\] 
Since the above is true for each $i\in\intgR{1}{m}$, the claim follows.}  
\smallskip

As mentioned in the first paragraph, the claim implies that the trace of $G'_{\Psi}$ at $a$ is equal to $n$.
Notice, since $\#(\C\setminus\OM)\geq 2n$, by Result~\ref{Res:kobhypsymprod} we know $\Sigma^n(\Omega)$ is a
Kobayashi complete domain and hence it is taut. So part~$(a)$ of Result~\ref{Res:CCT} implies that each eigenvalue of $G'_{\Psi}(a)$ lies in $\overline{\D}$. These two facts together imply that each eigenvalue of $G'_{\Psi}(a)$ is equal to $1$.
By Result~\ref{Res:iterconv}, the sequence $\{G^{k}_{\Psi}\}$ converges in $\hol(\Sigma^n(\Omega),\,\Sigma^n(\Omega))$ to a holomorphic retraction
$\rho:\Sigma^n(\Omega)\lrarw\Sigma^n(\Omega)$. Of course, the holomorphic retract $\rho(\Sigma^n(\Omega))$, 
which is a closed submanifold of $\Sigma^n(\Omega)$, is fixed point-wise by the map $G$. 
Now, by Corollary~\ref{Cor:unitspacetangspaceoflimitmani}, we know that the dimension of the retract is equal 
to the number of eigenvalues of $G'_{\Psi}(a)$ that belong to the boundary of $\D$. Since the latter is equal to $n$, 
it follows that $\rho(\Sigma^n(\Omega))=\Sigma^n(\Omega)$
whence it follows that $G_{\Psi}$ is the identity map on $\Sigma^n(\Omega)$.
This establishes the conclusion of the theorem when $A$ is diagonalizable.%
\medskip

\noindent{\bf{When $A$ is non-derogatory:}}
Since the matrix $A$ is non-derogatory, $A$ is similar to
the companion matrix of its characteristic polynomial; see \cite[p.195]{hornJohn:matanal85}.
Also, recall the map $\kappa:\Sigma^n(\Omega)\lrarw\Sn$, a right inverse of the map $\ch$,
as in the proof of Lemma~\ref{L:indmapsymprod}. Using $\kappa$, if $a=(a_1,\dots,a_n)=\ch(A)$ then $A$ is non-derogatory if and only if $A$
is similar to
\[
\kappa(a)=
\begin{bmatrix}
\ 0  & {} & {} & -a_k \ \\
\ 1  & 0  & {} & -a_{k-1} \ \\
\ {} & \ddots  & \ddots & \vdots \ \\
\ \text{\LARGE{0}} &   & 1 & -a_{1} \
\end{bmatrix}_{n\times n}.
\]
Owing to Lemma~\ref{L:probconjinva}, we can assume that $A=\kappa(a)=\kappa(\ch(A))$.
Recall that the map $G_{\Psi}$ associated with the map $\Psi$ is given by $G_{\Psi}=\ch\circ\Psi\circ\kappa$
(see the proof of Lemma~\ref{L:indmapsymprod}).
Note that $G_{\Psi}(a)=a$ and $G'_{\Psi}(a)=\ch'(\Psi(\kappa(a)))\circ\Psi'(\kappa(a))\circ\kappa'(a)$. Since
$\Psi(A)=A$ and $\Psi'(A)=\I$, we get $G'_{\Psi}(a)=\ch'(A)\circ\kappa'(a)$.
The map $\kappa$ is a right inverse of $\ch$ passing through $A$, thus $\ch'(A)\circ\kappa'(a)=\I$ whence it 
follows that $G'_{\Psi}(a)=\I$. We now invoke part~$(b)$ of Result~\ref{Res:CCT} to 
conclude that $G_{\Psi}$ is the identity map on $\Sigma^n(\Omega)$.
\end{proof}

We end this section with the proof of Corollary~\ref{cor:locconj} but before that we state a result. 
\begin{result}[Baribeau--Ransford, \cite{BaribeauRansford:specpres00}]\label{Res:BariRansford}
Let $\mathcal{U}$ be an open subset of $M_n(\C)$ and let $\Psi:\mathcal{U}\lrarw M_n(\C)$ be a 
spectrum-preserving \,$\mathcal{C}^1$-diffeomorphism of \,$\mathcal{U}$ onto $\Psi(\mathcal{U})$. Then 
$\Psi(W)$ is conjugate to $W$ for any $W\in\mathcal{U}$.
\end{result}
Also, see \cite[Th\'{e}or\`{e}me~2]{BaribeauRoy:speccharofJordanform00} for an analytic version of the above result. We now present

\begin{proof}[The proof of Corollary~\ref{cor:locconj}]
Notice, by Theorem~\ref{T:mainT1}, $\Psi$ is spectrum-preserving. Since $\Psi'(A)=\I$, by the inverse function theorem, there exists a neighbourhood $\mathcal{N}$ of $A$ such that
$\Psi:\mathcal{N}\lrarw\Psi(\mathcal{N})$ is a biholomorphism. Now, the corollary follows
from Result~\ref{Res:BariRansford}. 
\end{proof}

\section{Computation of the rank of $\ch'(A)$}\label{S:rankcomp}
In this section, we shall present the proof of Proposition~\ref{P:rankofcharmap}. As we shall see, a key tool in our proof is the local decomposition of 
the map $\ch$ as described in Lemma~\ref{L:localdecom}. 
In what follows, given integers $j< k$, ${\intgR{j}{k}}^2$ denotes the cartesian product of ${\intgR{j}{k}}$ 
with itself. We begin with the case when the matrix $A$ is a nilpotent matrix.

\begin{lemma}\label{L:ranknil}
Given $p\in\nat,\,p\geq 2$, let $A\in M_p(\C)$ be a nilpotent matrix. Then the rank of $\ch'(A)$ is equal to the degree of the minimal polynomial of $A$. 
\end{lemma}

\begin{proof} Note that for any $S\in M_p(\C)$ that is invertible, we have $\ch'(A)H=\ch'(S^{-1}AS)(S^{-1}HS)$.
It follows from this that the rank of $\ch'$ is similarity invariant. Hence
we shall assume $A$ to be in a Jordan canonical form. More precisely, we write
\[
 A=J_1(0)\oplus\dots\oplus J_m(0),
\]
where for each $i\in\intgR{1}{m},\,J_i(0)$ is a Jordan block of size $r_i$ with eigenvalue $0$
such that $r_1\leq r_2\leq\dots\leq r_m$.
Observe the degree of the minimal polynomial of $A$ is $r_m$. For each pair 
of indices $(j,\,k)\in{\intgR{1}{p}}^2$, let us denote by $E_{j,\,k}\in M_{p}(\C)$ the matrix
whose $(j,\,k)$-th entry is $1$ and every other entry is $0$.
\smallskip

\noindent{\bf Claim.} Fix a $(j,k)$ such that $(j,k)\notin{\intgR{r_i+1}{r_i+r_{i+1}}}^2$ for every $i\in\intgR{0}{m-1}$,
and where $r_0:=0$. Then $\ch(A+\eps\,E_{j,\,k})=\ch(A)$.
\smallskip

\noindent To see the claim, first consider the case when $A$ consists of only two Jordan blocks, i.e.,
$m=2$. In this case, if we change the matrix $A$ to $A+\eps\,E_{j,\,k}$ with $(j,\,k)$ as in the claim
then the Jordan blocks are unaffected and only one of the cross-diagonal blocks
of $A$\,---\,which are $0$ blocks\,---\,gets changed. The claim now easily follows from the following 
formula for the determinant of a block matrix:
\[
  \det\left[\begin{array}{ c c }
    A & B \\
    
    0 & D
  \end{array}\right]
  =\det(A)\,\det(D)=
 \det \left[\begin{array}{ c  c }
    A & 0 \\
    
    C & D
  \end{array}\right].
\]
The general case now follows from the principle of mathematical induction on the number of blocks,
together with the formulas above. The upshot of the above claim is that for each
$(j,\,k)\notin\intgR{r_i+1}{r_i+r_{i+1}}^2$ for every $i\in\intgR{0}{m-1}$,
the matrices $E_{j,\,k}$ belong to the kernel of $\ch'(A)$. 
\smallskip

Now, let $(j,k)\in\intgR{r_{k_0}+1}{r_{k_0}+r_{k_0+1}}^2$ for some fixed $k_0\in\intgR{0}{m-1}$.
Then the perturbed matrix $A+\eps E_{j,k}$ is such that the only block of $A$ that gets changed is the
$(k_0+1)$-th Jordan block of size $r_{k_0+1}$. Now, by the very definition of $\ch$, we have
\[
\det\big(t\I-(A+\eps\,E_{j,\,k})\big)=t^p+\sum_{\nu=1}^{p}(-1)^{\nu}\ch_{\nu}(A+\eps\,E_{j,\,k})\,t^{p-\nu}.
\]
On the other hand, we have
\[
\det\big(t\,\I-(A+\eps\,E_{j,\,k})\big)=t^{(p-r_{k_0+1})}\,\det\big(t\,\I-J_{k_0+1}(\eps)\big),
\]
where $J_{k_0+1}(\eps)$ is obtained from $J_{k_0+1}(0)$ by adding $\eps$ to its $(j,k)$-th entry
and keeping every other entry fixed. Observe that $\det\big(t\,\I-J_{k_0+1}(\eps)\big)$
is a monic polynomial of degree $r_{k_0+1}$. Hence the coefficient of the term $t^{p-\nu}$ in 
$\det\big(t\,\I-(A+\eps\,E_{j,\,k})\big)$, when
$\nu>r_{k_0+1}$, are all $0$. This implies that $\ch(A+\eps\,E_{j,k})$ is a point in $\C^{p}$, whose $\nu$-th
coordinate is zero for every $\nu>r_{k_0+1}$.
Therefore, for each $k_0\in\intgR{0}{m-1}$,
 and for each $(j,k)\in{\intgR{r_{k_0}+1}{r_{k_0}+r_{k_0+1}}}^2$,
$\ch(A+\eps\,E_{j,k})$ is
a point in $\C^{p}$, all of whose $\nu$-th coordinates are zero when $\nu>r_m$. This shows that 
${\rm rank}(\ch'(A))\leq r_m$. Next, we shall prove the converse of this inequality. 
\smallskip

Consider the matrix $B:=\oplus_{i=1}^m B_i$, where each $B_i$ is a companion matrix similar to $J_i(0)$.
Then $B$ is similar to $A$. Now consider
$H\in M_p(\C)$ and write $H=[H_1,\dots,H_p]$, where $H_i$'s are columns of $H$. We choose $H$
in such a way that $H_i$'s are all zero columns when 
$i\neq p$ and writing $H_p=(h_{p},\dots, h_{1})^{T}$, we have $h_{j}=0$ when $j>r_m$. Then we have 
\[
\det\big(t\,\I-(B+H)\big)=t^{p-r_m}\,\det\big(t\,\I-\widetilde{B}_m\big),
\]
where $\widetilde{B}_m$ is a companion matrix of order $r_m$ whose last column is the vector
$(h_{r_m},\dots,h_1)^T$. Hence we have 
$\det\big(t\,\I-\widetilde{B}_m\big)=t^{r_m}+\sum_{j=1}^{r_m}(-h_j)\,t^{r_m-j}$. This, together with the above equation implies 
\[
\det\big(t\,\I-(B+H)\big)=t^{p}+\sum_{j=1}^{r_m}(-h_j)\,t^{p-j}
\]
whence $\ch(B+H)=(h_1,\dots,(-1)^{r_m-1}h_{r_m},0,\dots,0)$. Now, since $\ch(B)=0\in\C^p$,
the subspace of $\C^p$, $\{z\in\C^p\,:\,z=(z_1,\dots,z_p)\,:\,z_j=0,\, j>r_m\}$ is contained in the range space of
$\ch'(B)$. Hence ${\rm{rank}}(\ch'(B))={\rm{rank}}(\ch'(A))\geq r_m$.
\end{proof}

The next lemma says that the rank of $\ch'$ is invariant under translation by a scalar matrix.

\begin{lemma}\label{L:rankinv}
Let $A\in M_n(\C)$ be given and let $\lam\in\C$ be a fixed complex number. Then the rank of $\ch'(A)$ is equal to the rank of $\ch'(A_{\lam})$, where
$A_{\lam}=A-\lam\,\I$. 
\end{lemma} 
\begin{proof} Choose a bounded domain $\OM$ that contains $\sigma(A)$ and fix it. Recall that
$\OM_{\lam}:=\{z-\lam\,:\,z\in\OM\}$. Then $L_{\lam}(\Sn)=
S_n(\Omega_{\lam})$, where $L_{\lam}(W)=W-\lam\,\I$ for all $W\in M_n(\C)$. Consider
$L_{-\lam}\in\hol(S_n(\Omega_{\lam}),\,\Sn)$, the inverse of $L_{\lam}$, which is a biholomorphism that 
maps $A_\lam$ to $A$. As noted earlier, (see Figure~\ref{F:comd3} in Subsection~\ref{SS:implemma}) there exists a holomorphic map $G_{-\lam}:\Sigma^n(\Omega_\lam)\lrarw
\Sigma^n(\Omega)$ such that
\[ 
G_{-\lam}\circ\ch=\ch\circ L_{-\lam}.
\]
In fact, $G_{-\lam}\big(\pi_n(z_1,\dots,z_n)\big)=\pi_n(z_1+\lam,\dots,z_n+\lam)$, where $\pi_n:\C^n\longrightarrow\C^n$ is the
symmetrization map. In partcular, $G_{-\lam}$ is a biholomorphism. Now taking the derivatives of both sides in 
the above equation at $A_{\lam}$ we get
\[
G'_{-\lam}(\ch(A_{\lam}))\circ\ch'(A_{\lam})=\ch'(A)\circ L_{-\lam}'(A_{\lam}).
\]
Notice that $ L_{-\lam}'(A_{\lam})=\I$ and hence $G'_{-\lam}(\ch(A_{\lam}))\circ\ch'(A_{\lam})=\ch'(A)$. Since 
$G'_{-\lam}(\ch(A_{\lam}))$ is an invertible linear transformation, it follows that that ${\rm{rank}}(\ch'(A_{\lam}))={\rm{rank}}(\ch'(A))$.
\end{proof}

We are now ready for

\begin{proof}[The proof of Proposition~\ref{P:rankofcharmap}] We shall work in the setting of 
Lemma~\ref{SS:locdecochi}.
Denote by $\lam_i,\,i\in\intgR{1}{m}$
the eigenvalues of $A$ having algebraic multiplicity $n_i$. 
Then Lemma~\ref{SS:locdecochi} says that there is a polydisc ${\sf{P}}(A;\,\delta)$ such that 
$\ch$ decomposes as $\ch\,=\,\tau\circ\theta$, where $\theta:{\sf{P}}(A;\,\delta)\lrarw\prod_{i=1}^m\sym{n_i}$
is a holomorphic open map defined by 
\eqref{E:deftheta} and $\tau:\prod_{i=1}^m\sym{n_i}\lrarw\sym{n}$ is as defined by \eqref{E:deftau}. 
Moreover, $\tau$ is a biholomorphism from
$\theta({\sf{P}}(A;\,\delta))$ onto $\ch({\sf{P}}(A;\,\delta))$.
Hence to compute the rank of $\ch'(A)$ it is sufficient to compute the rank of $\theta'(A)$. 
Let us write $\theta=(\theta_1,\dots,\theta_m)$, where $\theta_i:{\sf{P}}(A;\,\delta)\lrarw\sym{n_i}$ is a holomorphic map. Observe that, at the tangent space level, we have
\begin{align*}
 &\theta'(A):T_A\big({\sf{P}}(A;\,\delta)\big)\equiv\C^{n^2}\lrarw T_{\theta(A)}\Big(\prod_{i=1}^m\sym{n_i}\Big)
 \equiv\C^{n_1}\oplus\C^{n_2}\oplus\dots\oplus\C^{n_m} \ \text{and} \\
 & \theta'_i(A):T_A\big({\sf{P}}(A;\,\delta)\big)\equiv\C^{n^2}\lrarw T_{\theta_i(A)}\big(\sym{n_i}\big)\equiv \C^{n_i}
 \ \text{ for every $i\in\intgR{1}{m}$}. 
 \end{align*}
 \noindent As before, we assume $A$ to be in Jordan canonical form and write
 \[
  A=A_1\oplus\dots\oplus A_m, \ \ \ \text{where each $A_i=\oplus_{j=1}^{p_i} J_{r_{i,j}}(\lam_i)$}.
  \]
  Here, $J_{r_{i,\,j}}(\lam_i)$ is a Jordan block of size $r_{i,\,j}$ with eigenvalue $\lam_i$ such that
  $r_{i,\,1}\leq \dots\leq r_{i,\,p_i}$. 
  \medskip 
  
  \noindent{\bf Claim.} ${\rm{rank}}\,(\theta'(A))=\sum_{i=1}^m{\rm{rank}}\,(\theta'_i(A))$.
  \medskip
  
  \noindent Let us begin with defining certain subspaces of $M_n(\C)$. Recall that $E_{i,j}\in M_n(\C)$ is
  the matrix whose $(i,\,j)$-th entry is $1$ and every other entry is $0$. Now consider
  \begin{align*}
  \mathcal{S}_0:=&\,\text{span}\big\{E_{i,j}:\text{changing $A$ to $A+E_{i,j}$ does not change
   any of the blocks  $A_i$
  for any $i$}\big\} \\
  \mathcal{S}_k:=&\,\text{span}\big\{E_{i,j}:\text{changing $A$ to $A+E_{i,j}$ only results in a change in the block
  $A_k$}\big\}
  \end{align*}
  for every $k\in\intgR{1}{m}$. Note that $\mats_l\perp\mats_k$ whenever $0\leq l\neq k\leq m$ and 
  $\mats_0\oplus \mats_1\oplus\dots\oplus \mats_m\,=\,M_n(\C)=\,T_A\big({\sf{P}}(A;\,\delta)\big).$ 
\smallskip

\noindent{\bf Subclaim.} For every $i\in\intgR{1}{m}$, $\theta'_i(A)({\mats_i}^{\perp})=\,0$, i.e., $\theta'_i(A)(H)=0$ for all $H\in{\mats_i}^{\perp}$,
which is equivalent to 
\[
\lim_{\eps\to 0}\intf{}{}{\theta_i(A+\eps\,H)-\theta_i(A)}{\eps}\,=\,0 \ \ \forall H\in {\mats_i}^{\perp}.
\]
\smallskip

\noindent To prove the subclaim, we begin with noticing that
\[
 {\mats_i}^{\perp}\,=\,\mats_0\oplus\mats_1\oplus\dots\oplus\mats_{i-1}\oplus\mats_{i+1}\oplus\dots\oplus\mats_m.
 \]
If $H\in\mats_0$, following the same idea as in the proof of
Lemma~\ref{L:ranknil}, we see that $\theta(A+\eps\,H)\,=\,\theta(A)$ for every $\eps>0$ that is sufficiently 
small. This, in particular, implies $\theta_i(A+\eps\,H)\,=\,\theta_i(A)$ for every $H\in\mats_0$. Hence
$\theta'_i(A)(H)\,=\,0$ for all $H\in\mats_0$.
\smallskip

Now, let $H\in\mats_k$, $k\in\intgR{1}{m}$, $k\neq i$. For sufficiently small $\eps>0$,
$A+\eps\,H\in{\sf{P}}(A;\,\delta)$.
Further only the $k$-th diagonal block of $A$ is perturbed when we change $A$ to $A+\eps\,H$. 
This implies the value of $\theta_i$ under such perturbation is unaffected, i.e.,
$\theta_i(A+\eps\,H)=\theta_i(A)$ for all $H\in\mats_k$, $k\in\intgR{1}{m},\,k\neq i$. Hence 
$\mats_k\subset \text{Ker}(\theta'_i(A))$ for every $k\in\intgR{0}{m},\, k\neq i$. Since 
$\text{Ker}(\theta'_i(A))$ is a subspace of $M_n(\C)$, it follows that
${\mats_i}^{\perp}\subset\text{Ker}(\theta'_i(A))$ for every $i\in\intgR{1}{m}$. Hence
\[
\big[\theta'_i(A)\big]_{n_i\times n^2}\,=\,
\begin{pmatrix}
  [0]_{n_i\times n_0} &
  [0]_{n_i\times {n_1}^2} &
  \dots &
  [\theta'_i(A)|_{\mats_i}]_{n_i\times {n_i}^2}&
  \dots &
  [0]_{n_i\times{n_m}^2} &
\end{pmatrix}
\]
Here $n_0$ is the dimension of the subspace $V_0$. So from the form of the matrix
\[
\big[\theta'(A)\big]_{n\times n^2}\,=\,
\begin{pmatrix}
  [\theta'_{1}(A)]\\
  [\theta'_{2}(A)] \\
  \vdots \\
  [\theta'_{m}(A)]\\
\end{pmatrix},
\]
we see that the non-zero block of matrices $\big[\theta'_i(A)\big]_{n_i\times n^2}$ shifts rightward when $i$ 
increases from $1$ to $m$. This, in particular, implies that the rank of the matrix
$\big[\theta'(A)\big]_{n\times n^2}$ is equal to the sum of the rank of the matrices 
$[\theta'_i(A)|_{\mats_i}]_{n_i\times {n_i}^2}$ whence our claim follows. 
\smallskip

Now fix an $i\in\intgR{1}{m}$, we shall now show that ${\rm{rank}}(\theta'_{i}(A))=r_{i,\,p_i}$. Because of Lemma~\ref{L:rankinv},
we can assume without loss of generality that $\lam_i=0$. Moreover, from the above 
discussion, it is clear that 
\[
\theta'_i(A)(M_n(\C))=\theta'_i(\mats_i\oplus{\mats_i}^{\perp})\,=
\,\theta'_i(\mats_i).\]
Notice
that if we take a matrix $H\in\mats_i$ and decompose it in blocks corresponding to the blocks of $A$ as above
then the only non-zero block is the $i$-th diagonal block of size $n_i$. With this observation in hand,
proceeding exactly as in the proof of Lemma~\ref{L:ranknil}, we see that the rank of $\theta'_i(A)\,=\,
r_{i,\,p_i}$. Since the choice of $i$ was arbitrary, using the claim above we get that
the rank of 
$\theta'(A)=\,\sum_{i=1}^{m}r_{i,\,p_i}=\text{degree of the minimal polynomial of $A$}$. 
\end{proof}

See \cite{NTZ:DiscontLempKRmetspecball08} for a simple proof of Proposition~\ref{P:rankofcharmap}
when $A$ is a non-derogatory matrix.%

\section{The proof of Theorem~\ref{T:fixedpoint} and the case of $3\times 3$ matrices}\label{S:nonderog}
We present the proof of Theorem~\ref{T:fixedpoint} in this section. As mentioned in the introduction,
a key result in the proof of this theorem is Proposition~\ref{P:rankofcharmap}. The other tool is a result due to
Vigu\'{e} stated in Section~\ref{S:Basicprelim}.

\begin{proof}[The proof of Theorem~\ref{T:fixedpoint}] We consider the map $G_{\Psi}$ associated to the map $\Psi$ as in
Lemma~\ref{L:indmapsymprod}. The map $G_{\Psi}$ satisfies
$G_{\Psi}\circ\ch=\ch\circ\Psi$. Since $\Psi(A)=A$, we see that $G_{\Psi}(\ch(A))=\ch(A)$.
Write $a=\ch(A)$. Then, by Result~\ref{Res:Vig},
the fixed-point set of the map $G_{\Psi}$ denoted by ${\rm{Fix}}(G_{\Psi})$
is a closed complex submanifold such that
\[
 T_{a}({\rm{Fix}}(G_{\Psi}))\,=\,\{\,\xi\in T_{a}(\sym{n})\,:\,G_{\Psi}'(a)\,\xi=\xi\,\}.
 \]
Differentiating both sides of $G_{\Psi}\circ\ch=\ch\circ\Psi$ at $A$ gives $G'_{\Psi}(a)\circ\ch'(A)=\ch'(\Psi(A))\circ\Psi'(A)$. Substituting
$\Psi(A)=A$ and $\Psi'(A)=\I$, we get  $G'_{\Psi}(a)\circ\ch'(A)=\ch'(A)$. In particular, every $\xi\in T_{a}(\sym{n})$ that belongs to the range 
space of $\ch'(A)$ is in $T_{a}({\rm{Fix}}(G_{\Psi}))$. Hence
 \[\dim_{\C}[T_{a}({\rm{Fix}}(G_{\Psi}))]\geq{\rm{rank}}(\ch'(A)).\]
By Proposition~\ref{P:rankofcharmap}, the rank of $\ch'(A)$ is equal to the degree of the minimal polynomial of $A$. This implies that the dimension of
the fixed-point set ${\rm{Fix}}(G_{\Psi})$ is greater than or equal to the degree of the minimal polynomial of $A$. Clearly, if $W\in\Sn$ be such that
$\ch(W)\in{\rm{Fix}}(G_{\Psi})$ then $\ch(W)=\ch(\Psi(W))$. Taking
$\mathscr{S}={\rm{Fix}}(G_{\Psi})$,
 Theorem~\ref{T:fixedpoint} follows.
\end{proof}

\subsection*{The case when $A$ is a $3\times 3$ matrix}
Let $\OM\subset\C$ be a domain such that $\#(\C\setminus\OM)\geq 6$. Consider $\Psi\in\hol\big(S_3(\OM),\,S_3(\OM)\big)$ such that 
$\Psi(A)=A$ and $\Psi'(A)=\I$.
To analyse this case completely, we present a lemma. For this,
let $\Omega$, $A$, $\Psi$ be as in the statement of Theorem~\ref{T:mainT1}. Choose a $\lam\in\C$ and consider $\Omega_{\lam}=
\{z-\lam:z\in\Omega\}$ as defined before. Note, if $\Omega$ satisfies the cardinality condition then so does
$\Omega_{\lam}$. Now recall $\Psi_{\lam}\in\hol(S_n(\Omega_{\lam}),\,S_n(\Omega_{\lam}))$ defined by
\[
\Psi_{\lam}=L_{\lam}\circ\Psi\circ L_{-\lam},
\]
where $L_{\lam}(W):=W-\lam\mathbb{I}$ and $L_{-\lam}$ is the inverse of $L_{\lam}$. Note, $\Psi_{\lam}(A_\lam)=A_{\lam}$ and $\Psi'(A_\lam)=\mathbb{I}$, where $A_\lam:=A-\lam\mathbb{I}$.

\begin{lemma}\label{L:transscale}
Let $G_{\Psi_\lam}\in\hol\big(\Sigma^n(\Omega_\lam),\,\Sigma^n(\Omega_\lam)\big)$ be the map associated to the map $\Psi_\lam$ as given 
by Lemma~\ref{L:indmapsymprod}. Then $G_{\Psi_\lam}$ is the identity map on $\Sigma^n(\Omega_\lam)$ if and only if $G_{\Psi}$ is 
the identity map on $\Sigma^n(\Omega)$.
\end{lemma}
\begin{proof}
It is easy to see that $G_{\Psi_\lam}=G_{\lam}\circ G_{\Psi}\circ G_{-\lam}$, where
$G_{-\lam}\big(\pi_n(z_1,\dots,z_n)\big)=\pi_n(z_1+\lam,\dots,z_n+\lam)$,
where $\pi_n:\C^n\longrightarrow\C^n$ is the
symmetrization map.
As noted before, $G_{-\lam}$ is a biholomorphism with the inverse
$G_{\lam}$ defined by $G_{\lam}\big(\pi_n(z_1,\dots,z_n)\big)=\pi_n(z_1-\lam,\dots,z_n-\lam)$.
Thus the lemma follows.
\end{proof}

We now analyse the case of order $3$ matrices. First,
using Lemma~\ref{L:probconjinva}, we shall assume $A$ to be in the Jordan canonical form.
Also, applying Lemma~\ref{L:transscale}, without loss of generality, we can assume that $0\in\sigma(A)$.
Moreover, if we let
\[
n_A(\lam)\,:=\,\text{the number of Jordan blocks corresponding to the eigenvalue $\lam$}
\]
then we can assume that $n_A(0)\geq n_A(\lam)$, where $\lam\in\sigma(A)$ and $\lam\neq 0$.
Notice that in the case when 
$1\,=\,n_A(0)\geq n_A(\lam)$, $A$ is non-derogatory and we know that $\Psi$ is spectrum-preserving by 
Theorem~\ref{T:mainT1}. In the case when $n_A(0)=n$, $A$ is the zero matrix and in this case too we are done by Theorem~\ref{T:mainT1}.
In the case $n=3$, we are only left to the case $n_A(0)=2$. Let $r_1$, $r_2$ be the sizes of these 
two blocks with $r_1\leq r_2$. The only choices that we are left with are
following:
\begin{itemize}
\item[$(a)$] The case $r_1=r_2=1$. Then
\[
A\,=\,
\begin{pmatrix}
0 & 0 & 0\\
0 & 0 & 0\\
0 & 0 & \lam\\
\end{pmatrix} \ \ \text{and} \ \ \lam\neq 0.
\]
Since $A$ is diagonal, it follows from Theorem~\ref{T:mainT1} that $\Psi$ is spectrum-preserving.
\smallskip

\item[$(b)$] The case $r_1=1,\ r_2=2$. Then
\[
A\,=\,
\begin{pmatrix}
0 & 0 & 0\\
0 & 0 & 1\\
0 & 0 & 0\\
\end{pmatrix}.
\]
Notice that $A$ is a nilpotent matrix of order $2$. The degree of the minimal polynomial for $A$ is $2$. 
In this case, by Theorem~\ref{T:fixedpoint} we get that
$G_{\Psi}$ is the identity map on a closed complex submanifold of 
$\Sigma^n(\Omega)$ of complex dimension at least $2$.
We are not able to obtain any further information about $G_{\Psi}$ in this case.
\end{itemize}

\section{Appendix: $\ch$ is an open map}\label{S:propchi}
In this section, we prove that the map $\ch$ that appears in the article is an open map. This is Proposition~\ref{P:chiopen} below. 
Not only this could be of independent interest to the reader but also
in Subsection~\ref{SS:locdecochi}, we refer to the proof of
this proposition to conclude that the map $\theta$ is an open map. Before we present our proof, we need the following result. 

\begin{result}\label{Res:contiroots}
Let $p_1,\,p_2$ be two polynomials of degree $n$ of the form
\begin{align*}
p_1(t)&=a_n+a_{n-1}t+\dots+a_1t^{n-1}+t^n,\\
p_2(t)&=b_n+b_{n-1}t+\dots+b_1t^{n-1}+t^n.
\end{align*}
Let ${\sf T}:=\max\big\{1,|a_1|,|b_1|,|a_2|^{1/2}, |b_2|^{1/2},\dots,|a_n|^{1/n},|b_n|^{1/n}\big\}$. If $\alpha_j,\,j\in\intgR{1}{n}$ are the roots of $p_1$ then
there is an ordering of the roots of $p_2$, $\beta_1,\beta_2,\dots,\beta_n$, such that
\[
|\beta_j-\alpha_j|\leq 4\,n\,{\sf T}\,{(||a-b||_2)}^{\frac{1}{n}},
\]
where $||a-b||_2:=\big(\sum_{j=1}^n|a_j-b_j|^{2}\big)^{1/2}$.
\end{result}

\noindent The above result is due to A.\,M.\,Ostrowski; see, for instance, \cite{BB:rootsofpoly99}.
Now, we prove

\begin{proposition}\label{P:chiopen}
The map $\ch:M_n(\C)\lrarw\C^n$ defined by
$\ch(B):=b=(b_1,\dots,b_n)$, where $b$ is such that the polynomial $t^n+\sum_{k=1}^n(-1)^{k}\,b_k\,t^{n-k}$ is the characteristic polynomial of 
$B$, is an open map.
\end{proposition}
\begin{proof}
Let $\mathcal{U}\subseteq M_n(\C)$ be a non-empty open set. Consider $\ch(\mathcal{U})$ and let 
$x\in\ch(\mathcal{U})$ be a fixed point. We shall show that there exists an $\eps>0$
such that $\mathbb{B}(x,\,\eps):=\{y\in\C^n: ||x-y||_2<\eps\}\subset\ch(\mathcal{U})$.
Choose $X\in\mathcal{U}$ such that $\ch(X)=x$ and write
\begin{equation}\label{E:A_0}
X=S\,\big({\rm{diag}}(\lam_1,\dots,\lam_n)+U\big)\,S^{-1},
\end{equation}
where $\lam_i$'s are eigenvalues of $X$, $S\in M_n(\C)$ is an invertible matrix and $U$ is a strictly upper triangular matrix.
\smallskip

Fix an $r\in(0,\,1)$ and consider $\mathbb{B}(x,\,r)$. Then for any $y\in\mathbb{B}(x,\,r)$,
by Result~\ref{Res:contiroots}, 
there exist complex numbers
$\mu_1,\dots,\mu_n$\,---\,that are roots of the 
polynomial $t^n+\sum_{k=1}^n\,{(-1)}^{k}\,y_k\,t^{n-k}$\,---\,such that
\[
|\lam_j-\mu_j|\leq\,4\,n\,{\sf T}\,{r}^{1/n},
\]
where ${\sf T}=\max\big\{1,|y_1|,|x_1|,|y_2|^{1/2}, |x_2|^{1/2},\dots,|y_n|^{1/n},|x_n|^{1/n}\big\}$ and $x_j,\,y_j$,
$j\in\intgR{1}{m}$, are the co-ordinates of $x$ and $y$ respectively. Since $|y_j-x_j|\leq r<1$,
we have $|y_j|\leq |x_j|+1$ for every $j\in\intgR{1}{m}$.
This shows that the constant ${\sf T}$ could be chosen in a way so that it only depends on $x$ and $n$.
\smallskip

Now consider $Y:=S\,\big({\rm{diag}}(\mu_1,\dots,\mu_n)+U\big)\,S^{-1}$, where $S$ and $U$ are as in \eqref{E:A_0}. Then 
$Y-X=S\,\big({\rm{diag}}(\mu_1-\lam_1,\dots,\mu_n-\lam_n)\big)\,S^{-1}$. Hence 
\[
||Y-X||_{\rm{op}}\,\leq\,||S||_{\rm{op}}\,||S^{-1}||_{\rm{op}}\,\max\{|\mu_j-\lam_j|\,:\,j\in\intgR{1}{n}\}\,\leq\,
||S||_{\rm{op}}\,||S^{-1}||_{\rm{op}}\,4\,n\,{\sf T}\,{r}^{1/n}.
\]
Since $X\in\mathcal{U}$ and $\mathcal{U}$ is an open set, there exists a $\delta>0$ such that the set
$\{W\in M_n(\C):||W-X||_{\rm{op}}<\delta\}
\subset\mathcal{U}$. Now, we choose an $r$ such that
$||S||_{\rm{op}}\,||S^{-1}||_{\rm{op}}\,4\,n\,{\sf T}\,{r}^{1/n}<\delta$. Then
$Y\in\{W\in M_n(\C):||W-X||_{\rm{op}}<\delta\}\subset\mathcal{U}$. Note,
$\ch(Y)=y\in\mathbb{B}(x,\,\eps)$, where $\eps=r$. 
This proves that for any arbitrary $x\in\ch(\mathcal{U})$, there exists an $\eps>0$\,---\,depending only on $x$ and
$n$\,---\,such that $\mathbb{B}(x,\,\eps)\subset\ch(\mathcal{U})$. Hence $\ch(\mathcal{U})$ is an open set.
\end{proof}

\section*{Acknowledgments}
\noindent{The second author is supported by a postdoctoral fellowship from the Harish-Chandra Research Institute (HBNI), Prayagraj. 
The third author is supported by a scholarship from the National Board for Higher Mathematics
(Ref. No. 2/39(2)/2016/NBHM/R\&D-II/11411).

\end{document}